\documentclass[12pt]{amsart}

\usepackage{latexsym, amssymb, amsmath,ulem,soul,esint}

\usepackage{amsfonts, graphicx}
\usepackage{graphicx,color}
\usepackage{url}
\newcommand{\be}{\begin{equation}}
\newcommand{\ee}{\end{equation}}
\newcommand{\beq}{\begin{eqnarray}}
\newcommand{\eeq}{\end{eqnarray}}

\newtheorem{thm}{Theorem}[section]

\newtheorem{lma}{Lemma}[section]
\newtheorem{prop}{Proposition}[section]

\newtheorem{claim}{Claim}[section]
\theoremstyle{remark}
\newtheorem{rem}{Remark}[section]
\numberwithin{equation}{section}

\def\be{\begin{equation}}
\def\ee{\end{equation}}
\def\bee{\begin{equation*}}
\def\eee{\end{equation*}}

\def\lf{\left}
\def\ri{\right}

\newcommand{\Ric}{\mathrm{Ric}}

\def\Ric{\text{\rm Ric}}
\def\Rm{\text{\rm Rm}}

\def\na{\nabla}

\def\la{\langle}
\def\ra{\rangle}

\def\e{\varepsilon}
\def\a{{\alpha}}

\def\R{\mathbb{R}}

\begin{document}

\title{manifolds with small curvature concentration}

 \author[P.-Y. Chan]{Pak-Yeung Chan}
\address[Pak-Yeung Chan]{Department of Mathematics, University of California, San Diego, La Jolla, CA 92093}
\email{pachan@ucsd.edu}

 \author[S.-C. Huang]{Shaochuang Huang}
\address[Shaochuang Huang]{Department of Mathematics, Southern University of Science and Technology, Shenzhen, Guangdong, China.}
\email{huangsc@sustech.edu.cn}

\author[M.-C. Lee]{Man-Chun Lee}
\address[Man-Chun Lee]{Department of Mathematics, The Chinese University of Hong Kong, Shatin, N.T., Hong Kong}
\email{mclee@math.cuhk.edu.hk}


\date{\today}

\begin{abstract} 
In this work, we construct distance like functions with integral hessian bound on manifolds with small curvature concentration and use it to construct Ricci flows on  manifolds with possibly unbounded curvature. As an application, we study the geometric structure of those manifolds without bounded curvature assumption. In particular, we show that manifolds with Ricci lower bound, non-negative scalar curvature, bounded entropy, Ahlfors $n$-regular and small curvature concentration are topologically Euclidean. 
\end{abstract}


\maketitle

\section{Introduction}

Let $(M,g)$ be a complete manifold. One of the central themes in differential geometry is to determine how far is $M$ from being the standard space-form. The investigation under point-wise curvature control is successful. It is natural and important to consider generic geometric structure when the curvature $||\Rm||_{L^p}$ is only controlled in integral sense. This integral quantity, together with some mild volume non-collapsing conditions often plays a crucial role in the compactness theory and singularity analysis of the limit of sequence of Riemannian manifolds. In \cite{Cheeger2003}, Cheeger studied the problem when the manifold is in addition volume non-collapsed with Ricci lower bound. The technique is elliptic based on his foundational work with Colding \cite{CheegerColding1997}. When $p=n$ concerning the structure of complete  manifolds, it was shown by Cheeger \cite{Cheeger2003} that manifold with $\Ric\geq 0$, Euclidean volume growth and small curvature concentration must be diffeomorphic and pointed Gromov-Hausdorff close to the Euclidean space. The curvature concentration condition now serves as the major almost Euclidean information forcing certain rigidity.  On the other hand in \cite{Wang2020}, Wang considered the Ricci flow on manifolds with almost Euclidean volume condition. It was proved that if $(M,g_0)$ has $\Ric\geq 0$ and almost Euclidean volume growth, then one can construct a canonical diffeomorphism from $M$ to $\mathbb{R}^n$ by establishing long-time existence of Ricci flow. This gives a parabolic proof to the celebrated work of Cheeger-Colding theory. This seems to be natural. Motivated by this, it seems natural and tempting to find a geometric approach to the work of Cheeger \cite{Cheeger2003} using parabolic method which might have further generalization. 

Related regularization result using Ricci flow has been studied extensively \cite{Li-2012, Wang2011,Yang2011,Yang1992a,Yang1992b,Xu2013} where small critical $L^{n/2}$ norms of curvature is assumed while non-scaling invariant norm on Ricci ($||\Ric||_{p}$ where $p>\frac{n}{2}$) is required. Particularly, the boundedness of initial $||\Ric||_{L^p}$ enables one to control the point-wise $\Ric(g(t))$ for $t>0$ resulting a $C^0$ comparison between the resulting flow $g(t)$ and initial data $g_0$. This is crucial in the above mentioned work, it is however not likely the case in general. In \cite{ChanChenLee2021}, the first, third author and Chen took up the problem of generalizing the regularization of Ricci flow in the scaling invariant case. It was proved that when the initial metric has Ricci curvature and entropy bounded from below, then the bounded curvature Ricci flow will be instantaneously smoothed out if in addition the curvature concentration is small enough.  The curvature boundedness of the flow is important so that we can get the flow started using the by-now classical existence result of Shi \cite{Shi1989}. Various applications had been discussed under this technical assumption. The result is however not satisfactory in contrast with the elliptic method of Cheeger \cite{Cheeger2003} due to the need of bounded curvature. The ineffective curvature bounded is apparently only serving as a stepping stone, it is reasonable to expect that one can always start a Ricci flow without this technical assumption. The general existence theory of Ricci flow is however unclear so far except for some special cases.  

In this work, we continue the study of complete non-compact manifolds with small curvature concentration. Given its parabolic nature, the Ricci flow will in general turn almost Euclidean structures into useful estimates. This is usually called the pseudolocality of Ricci flows. The major challenge is to get the flow started with the parabolic epsilon regularity estimates from the almost Euclidean initial data. To do this, we employ the method of Topping \cite{Topping2010} to construct global solution without full curvature bound. The novel idea is to approximate the initial metric by bounded curvature complete manifolds in the local smooth sense. This method is of great success but usually requires the manifold to support a smooth distance like function with suitable control on hessian. This is usually related to the curvature control thanks to the Hessian comparison.

Despite the technical issue on the cut locus of the usual distance function, the missing of hessian control of any kind remains the essential difficulty. When the curvature is fully bounded, existence of distance like function with hessian bound was considered by Shi \cite{ShiPhd} for the purpose of developing Ricci flow maximum principle theory, see also \cite{CheegerGromov1985}. In this work, in order to adopt the setting in the content of $L^{n/2}$ bound of curvature, we develop an integral version of Shi's construction, see Theorem~\ref{Thm:exhaustion-C2}.  Precisely, we show that under small local curvature concentration relative to local Sobolev constant, manifolds with Ricci lower bound will support a distance like function such that its hessian is locally bounded uniformly on any unit ball in $L^n_{loc}$. Other than the importance on constructing Ricci flow, we use it to study a density problem of Sobolev space, see Theorem~\ref{densityThm}. 

The purpose of this work is to extend the applicability of earlier work \cite{ChanChenLee2021} where bounded curvature assumption is imposed to start the flow. This confirms the expectation on the Ricci flow existence. Our main result is as follows:

\begin{thm}\label{main-Existence}
For all $n\geq 4$, $A,L,\tau,\lambda,v_0>0$, there are $C_0,\sigma,T>0$ depending only on $n,A,L,\tau,\lambda,v_0$ such that the following holds. Suppose $(M^n,g_0)$ is a complete non-compact manifold such that for all $x\in M$, 
\begin{enumerate}
\item[(i)] $\mathcal{R}(g_0)\geq -\lambda$;
\item[(ii)] $\mathrm{Vol}_{g_0}\left( B_{g_0}(x,2r)\right)\leq L \cdot \mathrm{Vol}_{g_0}\left( B_{g_0}(x,r)\right)\leq Lv_0 r^n$ for all $r\leq 1$;
\item[(iii)] $\bar\nu(B_{g_0}(x,1),g_0,\tau)\geq -A$;
\item[(iv)] $\displaystyle\left(\int_{B_{g_0}(x, 1)}|\Rm(g_0)|^{n/2}d\mathrm{vol}_{g_0}\right)^{2/n}\leq \e $ for some $\e<\sigma $;
\item[(v)] $\liminf_{z\to +\infty}\Ric(z)>-\infty$
\end{enumerate}
Then there exists a complete Ricci flow $g(t)$ on $M\times [0,T]$ with $g(0)=g_0$ such that for any $x\in M$ and  $t\in (0,  T]$, we have 
\be 
\left\{
\begin{array}{ll}
\mathcal{R}(g(t))\geq -\lambda;\\
    |\Rm|(x,t)\leq{C_0\e}{t}^{-1};   \\
     \mathrm{inj}_{g(t)}(x) \geq C_0^{-1}\sqrt{t};\\
     \displaystyle\left(\int_{B_t(x,1)}|\Rm(g(t))|^{n /2}d\mathrm{vol}_{g(t)}\right)^{2/n}\leq C_0 \e. 
\end{array}
\right.
\ee
\end{thm}
The $\bar \nu$ entropy is used to control the log Sobolev inequality. The precise definition of $\bar \nu$ will be given in Section $2$. This is a quantitative version of short-time existence. The entropy here play the role to ensure the non-collapsing. In fact, with a Ricci lower bound, it will be equivalent to say that the volume ratio is bounded from below up to a scale, see \cite[Section 3]{Wang2018}.
Combining with the scalar curvature lower bound gives a lower bound of the Perelman's entropy which is commonly considered in Ricci flow theory. Together with a simple scaling argument, this implies a short-time existence if we only have $|\Rm(g_0)|\in L^{n/2}$, see Theorem~\ref{Ln/2-existence}. 

There is a large body of work on studying the topological rigidity of complete non-compact manifolds with $\Ric\geq 0$ and almost Euclidean structure such as almost Euclidean volume growth or Sobolev constant, for instances see \cite{CheegerColding1997, Chen2019,Ledoux1999,Xia1,Xia2} and the reference therein. With the newly established short-time existence, we may consequently generalize the applications in \cite{ChanChenLee2021}. In this regard, we prove the following. The relaxation from $\Ric\geq 0$ to $\mathcal{R}\geq 0$ might be of independent interest. 
\begin{thm}\label{Euc-gap}
For all $n\geq 4$ and $\e,A,L,v_0>0$, there is $\sigma>0$ depending only on $n,A,L,\e,v_0$ such that if $(M^n,g_0)$ is a complete non-compact manifold satisfying
\begin{enumerate}
\item[(i)] $\mathcal{R}(g_0)\geq 0$;
\item[(ii)] $\mathrm{Vol}_{g_0}\left( B_{g_0}(x,2r)\right)\leq L \cdot \mathrm{Vol}_{g_0}\left( B_{g_0}(x,r)\right)\leq Lv_0r^n$ for all $r>0$;
\item[(iii)] $\bar\nu(M,g_0)\geq -A$;
\item[(iv)] $\displaystyle\left(\int_{M}|\Rm(g_0)|^{n/2}d\mathrm{vol}_{g_0}\right)^{2/n}\leq \sigma $;
\item[(v)] $\Ric\geq -k$ for some $k\in \mathbb{R}$;
\end{enumerate}
for all $x\in M$. Then for all $x\in M$ and $r>0$, 
$$\mathrm{Vol}_{g_0}\left( B_{g_0}(x,r)\right)\geq (1-\e)\omega_n r^{n}.$$
Moreover, $M^n$ is homeomorphic to $\mathbb{R}^n$. If $n>4$, the homeomorphic can be chosen to be diffeomorphism.
\end{thm}

\begin{rem}
The assumption on the Ricci lower bound is only made to construct distance like function. It will be clear from the proof that it can be relaxed further, for instance see the treatment in \cite{He2016}. We leave it to the interested readers.
\end{rem}

Indeed, a lower bound on global entropy implies a lower bound on volume ratio, see \cite[Section 3]{Wang2018} for example. Hence, the volume growth assumption can be rephrased as uniformly Ahlfors $n$-regular for all scale, i.e. there exists $C(n,L,A,v_0)>1$ such that for all $x\in M$ and $r>0$, $$C^{-1}r^n\leq \mathrm{Vol}_{g_0}\left( B_{g_0}(x,r)\right)\leq Cr^n.$$

The result is largely motivated by a result of Cheeger \cite[Theorem 4.32]{Cheeger2003}. When applying to the metric cone at infinity, the result of Cheeger says that when $k=0$, then $(M,g,x_0)$ must be pointed Gromov-Hausdorff close to $\mathbb{R}^n$ and is diffeomorphic to $\mathbb{R}^n$. In Theorem~\ref{Euc-gap}, if we assume $\Ric\geq 0$, then $(M,g_0,p)$ is pointed Gromov-Hausdorff close to $(\mathbb{R}^n,g_{euc},0)$ by \cite[Theorem 0.8]{Colding1997}, see also \cite[Theorem A.1.11]{CheegerColding1997}. In fact, if $v_0$ is the almost Euclidean constant, then the argument in \cite[Section 5]{Wang2020} will imply that $d_{g_0}$ is pointed Gromov-Hausdorff close to the Euclidean space. This will give an alternative proof to Cheeger's result in this content.

There is also application on the regularity of the Gromov-Hausdorff limit. We refer readers to section~\ref{application-from-RF}.

The paper is organized as follows. In Section $2$, we shall show that the classical exhaustion function constructed by Schoen-Yau satisfies an integral Hessian bound using integration by part and Moser iteration method as in \cite{ShiPhd} by Shi. Some application of this integral estimate will also be given. In Section $3$, we recall (and refine) a Pseudolocality result proven by Chen, the first and third named authors in \cite{ChanChenLee2021}. The proof of Theorem \ref{main-Existence} will then be given in Section $4$. The short time existence of the Ricci flow is established by doing conformal change of the metric using the exhaustion function constructed (See also \cite{Topping2010}). Finally, we will cover some geometric applications of this flow existence result in Section $5$.  

{\it Acknowledgement}: S. Huang is partially supported by NSFC \#1200011128. M.-C. Lee is partially supported by RGC grant (Early Career Scheme) of Hong Kong No. 24304222 and a direct grant of CUHK.

\section{Exhaustion function with $L^n_{loc}$ control}\label{Sec:exhaustion}

In this section, we show that $M$ supports an exhaustion function $\rho\in C^\infty$ with bounded $L^n$ norm of $|\na^2\rho|$ over any unit balls under small local curvature concentration condition. When the curvature is globally bounded, the existence of such function was first constructed by Shi \cite{ShiPhd} where the hessian is uniformly bounded, see also \cite{Tam2010,ImperaRimoldiVeronelli2022}. 

 To formulate the result, let us recall the local entropy introduced by B. Wang \cite{Wang2018}: Let $\Omega$ be a connected domain with possibly empty boundary in $M$, 
\begin{align*}
&D_g(\Omega):=\left\{u: u\in W^{1,2}_0(\Omega), u\geq 0 \text{  and  } \|u\|_{L^2(\Omega)}=1 \right\},
\\
&W(\Omega, g, u, \tau):=\int_{\Omega}\tau(\mathcal{R}u^2+4|\nabla u|^2)-2u^2\log u \;   d\mathrm{vol}_g -\frac{n}{2}\log(4\pi\tau)-n,
\\
&\bar W(\Omega, g, u, \tau):=\int_{\Omega}4\tau |\nabla u|^2-2u^2\log u \;   d\mathrm{vol}_g -\frac{n}{2}\log(4\pi\tau)-n,
\\
&\nu(\Omega, g, \tau):= \inf_{u\in D_g(\Omega), s \in (0, \tau]}W(\Omega, g, u, s),
\\
&\bar\nu(\Omega, g, \tau):= \inf_{u\in D_g(\Omega), s \in (0, \tau]}\bar W(\Omega, g, u, s)
\end{align*}
where $\mathcal{R}$ denotes the scalar curvature of $(M,g)$.
In particular, the local entropy measures the optimal log-Sobolev constant over compact set and is closely related to the volume ratio, see \cite[Section 3]{Wang2018} for the detailed discussion. We also refer interested readers to \cite{GrigorSurvey,ZhangBook} for a comprehensive exposition on its relations to various functional constants.

In this section, we will consider manifolds those Ricci curvature is bounded from below and are non-collapsed in term of $\bar \nu$ entropy. This is equivalent to volume non-collapsing. We stick with the entropy assumption due to its significant in Ricci flow theory.

\begin{thm}\label{Thm:exhaustion-C2}
For any $A,\tau>0$ and $n\geq 5$, there exists $\delta(n,A,\tau)>0$ such that the following holds: Suppose $(M,g_0,x_0)$ is a complete non-compact manifold so that 
\begin{enumerate}
\item[(i)] $\bar \nu(B_{g_0}(x,1),g_0,\tau)\geq -A$;
\item[(ii)] $\Ric\geq -k$ for some $k\in \mathbb{R}$;
\item[(iii)] $\displaystyle \left(\int_{B_{g_0}(x,1)}|\Rm|^{n/2} \; d\mathrm{vol}_{g_0}\right)^{2/n}\leq \delta$
\end{enumerate}
for all $x\in M$. Then there is a smooth proper function $\rho$ on $M$ and a constant $C(n,k,A,\tau)>0$ such that for  $x\in M\setminus B_{g_0}(x_0,4)$, 
\begin{equation}\label{rho prop}
\left\{
\begin{array}{ll}
C^{-1}d_{g_0}(x,x_0){-C_1}\leq \rho(x) \leq d_{g_0}(x,x_0);\\
|\nabla\rho|^2+|\Delta \rho|\leq C;\\
 \displaystyle \int_{B_{g_0}(x,1)}|\nabla^2 \rho|^{n}\leq C,
\end{array}
\right.
\end{equation}
{where $C_1$ is a positive constant.}
\end{thm}

\begin{rem}
Since the function $\rho$ is smooth, the estimates will remain true on $M$ but the estimates around $x_0$ will largely depend on the geometry there.  The condition (i) together with lower bound on $\mathcal{R}$ will imply an explicit lower bound of $\nu(B_{g_0}(x,1),g_0,\tau)$, see \cite[Lemma 3.1]{Wang2018}. 
\end{rem}

 We start by recalling a construction by Schoen-Yau \cite{SchoenYau1994}.
\begin{lma}\label{SY-exhaustion}
Suppose $(M^n,g_0)$ is a complete non-compact manifold such that $x_0\in M$ and $\Ric(g_0)\geq -k$ for some $k\in \mathbb{R}$. Then there exist $C_0(n,k),\lambda(n,k)>0$ and a positive smooth function $h$ on $M$ such that on $M\setminus B_{g_0}(x_0,2)$,
\begin{equation}\label{proph}
\left\{
\begin{array}{ll}
\Delta h=\lambda h;\\
0<h(x)\leq C_1e^{-C_0d_{g_0}(x,x_0)};\\
|\nabla \log h|\leq C_0,
\end{array}
\right.
\end{equation}
where $C_1$ is a positive constant. In particular, $\rho=-\log h$ is a proper function on $M$ such that 
$$\sup_{M\setminus B_{g_0}(x_0,2)} |\nabla\rho|^2+|\Delta\rho|\leq C_0.$$
\end{lma}
\begin{proof}
This follows from the construction of Schoen-Yau \cite{SchoenYau1994} using eigenfunctions on $M\setminus B_{g_0}(x_0,1)$ and an interpolation with a cutoff function outside $B_{g_0}(x_0,2)$. 
\end{proof}

In what follows, we will use the strategy of Shi \cite{ShiPhd} to show that $\rho$ indeed has integral bound on Hessian under assumptions on entropy and small curvature concentration.  We begin with an observation that $\rho$ is uniformly bounded in $L^2_{loc}$ depending only on a Ricci lower bound.
\begin{lma}\label{ap-rho-L2}
Let $\rho$ be the function obtained from Lemma~\ref{SY-exhaustion}.  Then there is $C_1(n,k)$ such that for all $x\in M\setminus B_{g_0}(x_0,3)$ and $0<r\leq 1$,
$$\fint_{B_{g_0}(x,r)}|\nabla^2\rho|^2\; d\mathrm{vol}_{g_0}\leq C_1r^{-2}.$$
\end{lma}
\begin{proof}
In what follows, we will use $C$ to denote constants depending only on $n,k$. 
Let $\phi$ be a cutoff function on $B_{g_0}(x,2r)$ which is identical to $1$ on $B_{g_0}(x,r)$, vanishes outside $B_{g_0}(x,2r)$ and satisfies $|\nabla \phi|^2\leq Cr^{-2}\phi$.  For convenience, we denote 
$$S=\int_{B_{g_0}(x,2r)} \phi |\nabla^2\rho|^{2}\; d\mathrm{vol}_{g_0}.$$

By Stokes Theorem, we have 
\begin{equation}
    \begin{split}
   S        = \int_{B_{g_0}(x,2r)}  \phi  \rho_{ij} ^2\; d\mathrm{vol}_{g_0}
        &=-\int_{B_{g_0}(x,2r)}  \nabla_i \left( \phi  \rho_{ij} \right) \rho_j d\mathrm{vol}_{g_0}\\
        &\leq  \int_{B_{g_0}(x,2r)} - \phi_i \rho_{ij}\rho_j -\phi \rho_{ijj}\rho_i \; d\mathrm{vol}_{g_0}\\
        &=\mathbf{I}+\mathbf{II}.
    \end{split}
\end{equation}

We examine each of them one by one. 
\begin{equation}
    \begin{split}
        \mathbf{I}&\leq  \int_{B_{g_0}(x,2r)} |\nabla \phi| |\nabla \rho| |\nabla^2 \rho|\; d\mathrm{vol}_{g_0}\\
        &\leq\left( \int_{B_{g_0}(x,2r)} \phi |\nabla^2\rho|^2 \; d\mathrm{vol}_{g_0}\right)^\frac12\left(\int_{B_{g_0}(x,2r)} \frac{|\nabla\phi|^2}{\phi} |\nabla\rho|^2 \; d\mathrm{vol}_{g_0} \right)^\frac12 \\
        &\leq Cr^{-1}S^{1/2} |B_{g_0}(x,2r)|^{1/2}
    \end{split}
\end{equation}
where we have used the choice of $\phi$ and the gradient estimates.

For $\mathbf{II}$, we apply the Ricci identity 
to arrive at 
\begin{equation}
\begin{split}
\mathbf{II}&=-\int_{B_{g_0}(x,2r)}\phi \rho_{ijj}\rho_i \; d\mathrm{vol}_{g_0}\\
&=\int_{B_{g_0}(x,2r)}-\phi \langle \nabla \Delta \rho, \nabla \rho\rangle- \phi \Ric(\nabla \rho,\nabla \rho) \; d\mathrm{vol}_{g_0}\\
&\leq  k\int_{B_{g_0}(x,2r)}\phi  |\nabla\rho|^2 \; d\mathrm{vol}_{g_0}+ \int_{B_{g_0}(x,2r)}  \Delta\rho \cdot \mathrm{div}\left(\phi \nabla\rho \right)\;d\mathrm{vol}_{g_0}\\
&\leq  k\int_{B_{g_0}(x,2r)}\phi  |\nabla\rho|^2 \; d\mathrm{vol}_{g_0}+ \int_{B_{g_0}(x,2r)}  \Delta\rho \cdot  \left(\langle \nabla\phi,\nabla\rho\rangle+ \phi \Delta\rho \right)\;d\mathrm{vol}_{g_0}\\
&\leq C(n,k)r^{-1}|B_{g_0}(x,2r)|.
\end{split}
\end{equation}
Hence $S\leq C(n,k)r^{-2}|B_{g_0}(x,2r)|$. Result follows from the volume comparison and Ricci lower bound.
\end{proof}

When $n=4$, this Morrey type estimate is sufficient for the purpose of constructing Ricci flow. For $n>4$, we improve $L^2_{loc}$ in Lemma~\ref{ap-rho-L2} to $L^n_{loc}$ by refining the method of Shi \cite{ShiPhd}. We first need a Sobolev inequality under a lower bound on entropy.
\begin{lma}\label{ent to S ineq}
For any $A>1,\tau>0$, there is $C_S(n,A,\tau)>0$ such that if $(M^n,g)$ is a complete manifold for some $x\in M$,  we have 
$\bar\nu(B_{g}(x,1), g,\tau) \geq -A$, then for any $\varphi\in C^\infty_{c}(B_{g}(x,1))$, we have 
$$\left(\int_{B_{g}(x,1)} |\varphi|^\frac{2n}{n-2} \; d\mathrm{vol}_{g}\right)^\frac{n-2}{n}\leq C_S\int_{B_{g}(x,1)}|\nabla \varphi|^2+\varphi^2 \; d\mathrm{vol}_{g} .$$

\end{lma}
\begin{proof}
The proof is identical to \cite[Lemma 2.2]{ChanChenLee2021} which is based on \cite{Davies1989,Ye2015,Zhang2007}.
\end{proof}

Theorem~\ref{Thm:exhaustion-C2} is a consequence of Lemma~\ref{ent to S ineq}  and the following proposition:

\begin{prop}\label{prop:exhaustion-C3}
For any integer $n\ge 5$ and positive constants $C_S$, 
there exists $\delta(n,C_S)>0$ 
such that the following holds: Suppose $(M,g_0,x_0)$ is a complete non-compact manifold so that 
\begin{enumerate}
\item[(i)] For any $\varphi\in C^\infty_{c}(B_{g}(x,1))$
$$\left(\int_{B_{g_0}(x,1)} |\varphi|^\frac{2n}{n-2} \; d\mathrm{vol}_{g_0}\right)^\frac{n-2}{n}\leq C_S\int_{B_{g_0}(x,1)}|\nabla \varphi|^2+\varphi^2 \; d\mathrm{vol}_{g_0};$$
\item[(ii)] $\Ric\geq -k$ for some $k\in \mathbb{R}$ 
\item[(iii)] $\displaystyle \left(\int_{B_{g_0}(x,1)}|\Rm|^{n/2} \; d\mathrm{vol}_{g_0}\right)^{2/n}\leq \delta$
\end{enumerate}
for all $x\in M$. Then the function $\rho$ constructed in Lemma~\ref{SY-exhaustion} satisfies
\begin{equation}
\int_{B_{g_0}(x,1)}|\nabla^2 \rho|^n\; d\mathrm{vol}_{g_0}\leq C
\end{equation}
for any $x\in M\setminus B_{g_0}(x_0,4)$, where $C=C(n,k,C_S)>0$.
\end{prop}
Here we impose a weak form of $L^2$ Sobolev inequality since this follows from the entropy lower bound from Lemma~\ref{ent to S ineq}. 

\begin{proof} Throughout the proof, 
we use $c=c(n)$ to denote some dimensional constants and $C=C(n,k,C_S)$ to denote some constants depending on $n, k, C_S$,
their exact values may differ from line to line.
We shall apply the Moser iteration as in Shi's thesis \cite[Theorem 3.1]{ShiPhd} to obtain the integral bound for $|\nabla^2\rho|$. 
We start with an elliptic equation satisfied by $\nabla^2\rho$. By the Ricci identity and the second Bianchi identity, we have in any orthonormal frame:
\begin{eqnarray*}
\Delta \left(\na^2 \rho\right)_{ij}&=&\na^2\left(\Delta \rho\right)_{ij}+R_{ik} \left(\na^2 \rho\right)_{kj}+R_{jk} \left(\na^2 \rho\right)_{ki}-2 R_{ikjl}\left(\na^2 \rho\right)_{kl}\\
&&+\left(R_{ik,j}+R_{jk,i}-R_{ij,k}\right)\rho_k.\\
&=&2\rho_{ki}\rho_{kj}+2\rho_{kij}\rho_k+R_{ik} \left(\na^2 \rho\right)_{kj}+R_{jk} \left(\na^2 \rho\right)_{ki}-2 R_{ikjl}\left(\na^2 \rho\right)_{kl}\\
&&+\left(R_{ik,j}+R_{jk,i}-R_{ij,k}\right)\rho_k,
\end{eqnarray*}
where we have used $\Delta \rho =-\lambda+|\na \rho|^2$ in the last equality, and $\lambda$ is the constant in \eqref{proph}. Hence
\begin{eqnarray*}
\Delta|\na^2 \rho|^2&=&2|\na^3 \rho|^2+2\la\na^2 \rho, \Delta\na^2\rho\ra\\
&=& 2|\na^3 \rho|^2+2\left(R_{ik,j}+R_{jk,i}-R_{ij,k}\right)\rho_{ij}\rho_k\\
&&+\Rm\ast\na^2\rho\ast\na^2\rho+4\rho_{ij}\rho_{ki}\rho_{kj}+4\rho_{ij}\rho_{kij}\rho_k.
\end{eqnarray*}
Let $x\in M\setminus B_{g_0}(x_0,3)$, $1\ge r_1>r_2>0$, $w:=\sqrt{|\na^2 \rho|^2+1}\ge 1$ and $\phi$ any smooth nonnegative, compactly supported function on $B_{g_0}(x,1)$ such that $\phi$ vanishes outside $B_{g_0}(x,r_1)$, identical to $1$ on $B_{g_0}(x,r_2)$ and satisfies 
\be\label{gradphi}
|\na\phi|\le c\phi^{\beta} (r_1-r_2)^{-1}, 
\ee
where $\beta:=\frac{n-3}{n-2}\in(0,1)$ and $\beta\ge 1/2$. 
This can be found by compositing a linear function. 
Multiply the differential equation for $w$ by 
$-w^{p-2}\phi$ and integrate the equation over $B_{g_0}(x,1)$,


\begin{eqnarray*}
\int_{B_{g_0}(x,1)}-\phi w^{p-2}\Delta (w^2)\; d\mathrm{vol}_{g_0}&\le& -2\int_{B_{g_0}(x,1)}\phi w^{p-2}|\na^3 \rho|^2 d\mathrm{vol}_{g_0}\\
&&+c\int_{B_{g_0}(x,1)} \phi w^{p}|\Rm|\; d\mathrm{vol}_{g_0}\\
&&-2\int_{B_{g_0}(x,1)} \phi w^{p-2}\left(R_{ik,j}+R_{jk,i}-R_{ij,k}\right)\rho_{ij}\rho_k\; d\mathrm{vol}_{g_0}\\
&&-4\int_{B_{g_0}(x,1)}\phi w^{p-2}\rho_{ij}\rho_{ki}\rho_{kj}\; d\mathrm{vol}_{g_0}\\
&&-4\int_{B_{g_0}(x,1)}\phi w^{p-2}\rho_{ij}\rho_{kij}\rho_{k}\; d\mathrm{vol}_{g_0}\\
&=:& \mathbf{I}+\mathbf{II}+\mathbf{III}+\mathbf{IV}+\mathbf{V}
\end{eqnarray*}

where 
$p\ge p_0(n)$ and $p_0(n)>2$
is a dimensional constant to be determined. We also require 
$p\le n-2$. 
Using integration by part and Cauchy Schwarz inequality, we have,

\begin{eqnarray*}
\int_{B_{g_0}(x,1)}-\phi w^{p-2}\Delta (w^2)\; d\mathrm{vol}_{g_0}&=&
\frac{8(p-2)}{p^{2}}\int_{B_{g_0}(x,1)}|\sqrt{\phi}\na (w^{p/2})|^2\; d\mathrm{vol}_{g_0}\\
&&+\int_{B_{g_0}(x,1)}2\la\na\phi, \na w\ra w^{p-1}\; d\mathrm{vol}_{g_0}\\
&\ge&\frac{2(p-2)}{p^2}\int_{B_{g_0}(x,1)}\left|\na\left(\sqrt{\phi} w^{p/2}\right)\right|^2\; d\mathrm{vol}_{g_0}\\
&&-\frac{2p^2-4p+4}{p^2(p-2)}\int_{B_{g_0}(x,1)}\frac{|\na \phi|^2w^{p}}{\phi}\; d\mathrm{vol}_{g_0}\\
&\ge&\frac{2(p-2)}{p^2 C_S}\left(\int_{B_{g_0}(x,1)} \phi^{\frac{n}{n-2}}w^{\frac{np}{n-2}}\; d\mathrm{vol}_{g_0} \right)^{\frac{n-2}{n}}\\
&&-\frac{c(2p^2-4p+4)}{p^2(p-2)}\int_{B_{g_0}(x,r_1)}\frac{w^{p}}{(r_1-r_2)^2}\; d\mathrm{vol}_{g_0}\\
&&-\frac{2(p-2)}{p^2}\int_{B_{g_0}(x,1)} \phi w^{p}\; d\mathrm{vol}_{g_0}.
\end{eqnarray*}


\begin{eqnarray*}
\mathbf{II}&=&c\int_{B_{g_0}(x,1)} \phi w^{p}|\Rm|\; d\mathrm{vol}_{g_0}\\
&\le& c\left(\int_{B_{g_0}(x,1)}|\Rm|^{\frac{n}{2}}\; d\mathrm{vol}_{g_0}\right)^{\frac{2}{n}}\left(\int_{B_{g_0}(x,1)}\phi^{\frac{n}{n-2}}w^{\frac{np}{n-2}}\; d\mathrm{vol}_{g_0}\right)^{\frac{n-2}{n}}\\
&\le& c\delta\left(\int_{B_{g_0}(x,1)}\phi^{\frac{n}{n-2}}w^{\frac{np}{n-2}}\; d\mathrm{vol}_{g_0}\right)^{\frac{n-2}{n}}.
\end{eqnarray*}

We have used the smallness of $L^{n/2}$ norm of $|\Rm|$.
Similarly by the definition of $\rho$, \eqref{proph} and \eqref{gradphi},


\begin{eqnarray*}
\mathbf{III}&=&-2\int_{B_{g_0}(x,1)} \phi w^{p-2}\left(R_{ik,j}+R_{jk,i}-R_{ij,k}\right)\rho_{ij}\rho_k\; d\mathrm{vol}_{g_0}\\
&=& \int_{B_{g_0}(x,1)} \phi w^{p-2}\na \Ric\ast\na^2 \rho\ast\na \rho\; d\mathrm{vol}_{g_0}\\
&=& -\int_{B_{g_0}(x,1)} \na(\phi) w^{p-2}\Ric\ast\na^2 \rho\ast\na \rho\; d\mathrm{vol}_{g_0}\\
&&-\int_{B_{g_0}(x,1)} \phi \na(w^{p-2}) \Ric\ast\na^2 \rho\ast\na \rho\; d\mathrm{vol}_{g_0}\\
&&-\int_{B_{g_0}(x,1)} \phi w^{p-2} \Ric\ast\na^3 \rho\ast\na \rho\; d\mathrm{vol}_{g_0}\\
&&-\int_{B_{g_0}(x,1)} \phi w^{p-2} \Ric\ast\na^2 \rho\ast\na^2 \rho\; d\mathrm{vol}_{g_0}\\
&\le& \frac{C}{r_1-r_2}\int_{B_{g_0}(x,1)} \phi^{\beta}|\Rm|w^{p-1}\; d\mathrm{vol}_{g_0}+Cp\int_{B_{g_0}(x,1)} \phi |\Rm| w^{p-2} |\na^3\rho|\; d\mathrm{vol}_{g_0}\\
&&+c\int_{B_{g_0}(x,1)} \phi w^{p} |\Rm|\; d\mathrm{vol}_{g_0}\\
&\le& \frac{C}{r_1-r_2}\int_{B_{g_0}(x,1)} \phi^{\beta} |\Rm|w^{p-1}\; d\mathrm{vol}_{g_0}+Cp\int_{B_{g_0}(x,1)} \phi |\Rm| w^{p-2} |\na^3\rho|\; d\mathrm{vol}_{g_0}\\
&&+c\delta\left(\int_{B_{g_0}(x,1)}\phi^{\frac{n}{n-2}}w^{\frac{np}{n-2}}\; d\mathrm{vol}_{g_0}\right)^{\frac{n-2}{n}}.
\end{eqnarray*}

We then handle the terms in the last inequality separately. By virtue of Young's inequality,  $w\ge 1$ and $\beta= \frac{n-3}{n-2}$, 


\begin{eqnarray*}
&&\frac{C}{r_1-r_2}\int_{B_{g_0}(x,1)} \phi^{\beta} |\Rm|w^{p-1}\; d\mathrm{vol}_{g_0}\\
&\le&\frac{C}{r_1-r_2}\left(\int_{B_{g_0}(x,1)}  |\Rm|^{\frac{n}{2}}\; d\mathrm{vol}_{g_0}\right)^{\frac{2}{n}}\left(\int_{B_{g_0}(x,1)} \phi^{\frac{n\beta}{n-2}} w^{\frac{(p-1)n}{n-2}}\; d\mathrm{vol}_{g_0}\right)^{\frac{n-2}{n}}\\
&\leq &\frac{C\delta}{r_1-r_2}\left(\int_{B_{g_0}(x,1)} \phi^{\frac{n\beta}{n-2}} w^{\frac{(p-1)n}{n-2}}\; d\mathrm{vol}_{g_0}\right)^{\frac{n-2}{n}}\\
&\leq &\frac{C\delta}{(r_1-r_2)^{\frac{1}{1-\beta}}} \left[\mathrm{Vol}_{g_0}(B_{g_0}(x,r_1))\right]^{\frac{n-2}{n}}+\delta\left(\int_{B_{g_0}(x,1)} \phi^{\frac{n}{n-2}} w^{\frac{(p-1)n}{\beta(n-2)}}\; d\mathrm{vol}_{g_0}\right)^{\frac{n-2}{n}}\\
&\leq &\frac{C\delta}{(r_1-r_2)^{\frac{1}{1-\beta}}} \left[\mathrm{Vol}_{g_0}(B_{g_0}(x,r_1))\right]^{\frac{n-2}{n}}+\delta\left(\int_{B_{g_0}(x,1)} \phi^{\frac{n}{n-2}} w^{\frac{np}{(n-2)}}\; d\mathrm{vol}_{g_0}\right)^{\frac{n-2}{n}}.
\end{eqnarray*}

Using H\"{o}lder and Schwarz inequalities,


\begin{eqnarray*}
&&Cp\int_{B_{g_0}(x,1)} \phi |\Rm| w^{p-2} |\na^3\rho|\; d\mathrm{vol}_{g_0}\\
&\le& Cp\left(\int_{B_{g_0}(x,1)}\phi |\Rm|^2w^{p-2}
\; d\mathrm{vol}_{g_0} \right)^{\frac{1}{2}}\left(\int_{B_{g_0}(x,1)}\phi  w^{p-2}|\na^3\rho|^2
\; d\mathrm{vol}_{g_0} \right)^{\frac{1}{2}}\\
&\le&\frac{Cp^2}{\varepsilon}\int_{B_{g_0}(x,1)}\phi |\Rm|^2w^{p-2}\; d\mathrm{vol}_{g_0}+\varepsilon \int_{B_{g_0}(x,1)}\phi  w^{p-2}|\na^3\rho|^2
\; d\mathrm{vol}_{g_0}\\
&\le&\frac{Cp^2}{\varepsilon}\left(\int_{B_{g_0}(x,1)}|\Rm|^{\frac{n}{2}}\; d\mathrm{vol}_{g_0}\right)^{\frac{4}{n}}\left(\int_{B_{g_0}(x,1)}\phi^{\frac{n}{n-4}}w^{\frac{(p-2)n}{n-4}}\; d\mathrm{vol}_{g_0}\right)^{\frac{n-4}{n}}\\
&&+\varepsilon \int_{B_{g_0}(x,1)}\phi  w^{p-2}|\na^3\rho|^2
\; d\mathrm{vol}_{g_0}\\
&\le&\delta^2\left(\int_{B_{g_0}(x,1)}\phi^{\frac{n}{n-4}}w^{\frac{(p-2)n}{n-4}}\; d\mathrm{vol}_{g_0}\right)^{\frac{n-2}{n}}+\frac{C\, p^{n-2}\delta^2}{\varepsilon^{\frac{n-2}{2}}}+\varepsilon \int_{B_{g_0}(x,1)}\phi  w^{p-2}|\na^3\rho|^2
\; d\mathrm{vol}_{g_0}\\
&\le&\delta^2\left(\int_{B_{g_0}(x,1)}\phi^{\frac{n}{n-2}}w^{\frac{np}{n-2}}\; d\mathrm{vol}_{g_0}\right)^{\frac{n-2}{n}}+\frac{C\, p^{n-2}\delta^2}{\varepsilon^{\frac{n-2}{2}}}+\varepsilon \int_{B_{g_0}(x,1)}\phi  w^{p-2}|\na^3\rho|^2
\; d\mathrm{vol}_{g_0},\\
\end{eqnarray*}

we have used the facts that $w\ge 1$ and 
$p\le n-2$
in the last inequality. In conclusion, it holds that


\begin{eqnarray*}
\mathbf{III}&\le& c\delta\left(\int_{B_{g_0}(x,1)}\phi^{\frac{n}{n-2}}w^{\frac{np}{n-2}}\; d\mathrm{vol}_{g_0}\right)^{\frac{n-2}{n}}+\frac{C\delta}{(r_1-r_2)^{\frac{1}{1-\beta}}} \left[\mathrm{Vol}_{g_0}(B_{g_0}(x,r_1))\right]^{\frac{n-2}{n}}\\
&&+\frac{C\, p^{n-2}\delta^2}{\varepsilon^{\frac{n-2}{2}}}+\varepsilon \int_{B_{g_0}(x,1)}\phi  w^{p-2}|\na^3\rho|^2
\; d\mathrm{vol}_{g_0}.\\
\end{eqnarray*}

Similarly by \eqref{proph},


\begin{eqnarray*}
\mathbf{IV}&=&-4\int_{B_{g_0}(x,1)}\phi w^{p-2}\rho_{ij}\rho_{ki}\rho_{kj}\;d\mathrm{vol}_{g_0}\\
&\le&4\int_{B_{g_0}(x,1)}\phi_j w^{p-2}\rho_i\rho_{ki}\rho_{kj}\;d\mathrm{vol}_{g}+4\int_{B_{g_0}(x,1)}\phi w^{p-2}\rho_i\rho_{kij}\rho_{kj}\;d\mathrm{vol}_{g_0}\\
&&+4\int_{B_{g_0}(x,1)}\phi w^{p-2}\rho_i\rho_{ki}\rho_{kjj}\;d\mathrm{vol}_{g_0}
+C(p-2)\int_{B_{g_0}(x,1)} \phi w^{p-4}|\na^3\rho||\na^2 \rho|^3\;d\mathrm{vol}_{g_0}\\
&\le& \frac{C}{r_1-r_2}\int_{B_{g_0}(x,1)}\sqrt{\phi} w^{p}\;d\mathrm{vol}_{g_0}+Cp\int_{B_{g_0}(x,1)}\phi w^{p-1}|\na^3\rho|\;d\mathrm{vol}_{g_0}\\
&\le&\frac{C}{r_1-r_2}\int_{B_{g_0}(x,1)}\sqrt{\phi} w^{p}\;d\mathrm{vol}_{g_0}+
\varepsilon\int_{B_{g_0}(x,1)}\phi w^{p-2}|\na^3\rho|^2\;d\mathrm{vol}_{g_0}\\
&&+\frac{Cp^2}{\varepsilon}\int_{B_{g_0}(x,1)}\phi w^{p}\;d\mathrm{vol}_{g_0}.
\end{eqnarray*}


\begin{eqnarray*}
\mathbf{V}&=&-4\int_{B_{g_0}(x,1)}\phi w^{p-2}\rho_{ij}\rho_{kij}\rho_{k}\; d\mathrm{vol}_{g_0}\\
&\le& C\int_{B_{g_0}(x,1)}\phi w^{p-1}|\na^3\rho|\;d\mathrm{vol}_{g_0}\\
&\le& \varepsilon\int_{B_{g_0}(x,1)}\phi w^{p-2}|\na^3\rho|^2\;d\mathrm{vol}_{g_0}+\frac{C}{\varepsilon}\int_{B_{g_0}(x,1)}\phi w^{p}\;d\mathrm{vol}_{g_0}.
\end{eqnarray*}

By the volume comparison theorem and taking $\varepsilon=\frac{1}{4}$, $\delta>0$ sufficiently small depending on $n$ and $C_S$, 
we have the following reverse H\"{o}lder inequality


\be\label{RH}
\left(\int_{B_{g_0}(x,r_2)} w^{\frac{np}{n-2}}\; d\mathrm{vol}_{g_0} +1\right)^{\frac{n-2}{np}}\le\frac{C^{\frac{1}{p}}}{(r_1-r_2)^{\frac{1}{(1-\beta)p}}}\left(\int_{B_{g_0}(x,r_1)}w^{p}\;d\mathrm{vol}_{g_0}+1\right)^{\frac{1}{p}}.
\ee

Let $\mu:=\frac{n}{n-2}$. For $n\ge 5$, there exist unique 
$p_0=p_0(n)\in (2,2n/(n-2)]$
and integer $m=m(n)\ge 1$ such that 
$\mu^{m}p_0=n$.
It can be seen that 
$ p_0(n)\le 2n/(n-2)\le (n-2)$
for $n\ge 6$. Moreover, 
$p_0(5)=3=5-2\le 10/3$. For any $r\in (0,1]$ and $\theta\in (0,1)$, we let $\alpha:=(1-\theta)/(2-\theta)\in (0,1/2)$ and for $m+1\ge i\ge 0$,
\[
r_i:=2r-r\sum_{j=0}^{i}\alpha^{j}.
\]
It can be seen that $r_0=r$, $r_i$ is decreasing and $\lim_{i\to\infty} r_i=\theta r$. For any integer $i\ge 0$, we define a sequence of positive numbers as follows

\[
Q_i:=\left(\int_{B_{g_0}(x,r_i)}w^{p_0\mu^{i}}\;d\mathrm{vol}_{g_0}+1\right)^{\frac{1}{p_0\mu^{i}}}.
\]

By substituting $p=p_0\mu^{i}$, $r_1=r_i$ and $r_2=r_{i+1}$ into \eqref{RH} for $i=0, 1, \cdots, m-1$, we have

\[
Q_{i+1}\le \frac{C^{\frac{1}{p_0\mu^{i}}}}{r^{\frac{1}{(1-\beta)p_0\mu^{i}}}\alpha^{\frac{i+1}{(1-\beta)p_0\mu^{i}}}} Q_i.
\]

Since $\mu>1$, 
$\sum_{i=0}^{\infty}\frac{1}{p_0\mu^{i}}$ and $\sum_{i=0}^{\infty}\frac{i+1}{p_0\mu^{i}}$ converge. We may iterate the above inequality from $i=0$ to $m-1$ to get

\be\label{FIE}
\begin{split}
\left(\int_{B_{g_0}(x,\theta r)}w^{n}\;d\mathrm{vol}_{g_0}+1\right)^{\frac{1}{n}}&\le \frac{2^{\sum_{i=0}^{\infty}\frac{i+1}{(1-\beta)p_0\mu^{i}}}C^{\sum_{i=0}^{\infty}\frac{1}{p_0\mu^{i}}}}{r^{\sum_{i=0}^{\infty}\frac{1}{(1-\beta)p_0\mu^{i}}}(1-\theta)^{\sum_{i=0}^{\infty}\frac{i+1}{(1-\beta)p_0\mu^{i}}}}\\
&\quad\times\left(\int_{B_{g_0}(x,r)}w^{p_0}\;d\mathrm{vol}_{g_0}+1\right)^{\frac{1}{p_0}}.
\end{split}
\ee

We shall iterate $r$ and $\theta$ infinitely many times as in \cite[Chapter 19]{Li2012} so that we can replace $p_0$ by $2$ in \eqref{FIE}. Since $n\ge 5$,  we can fix a small $\delta=\delta(n)\in (0,1)$ such that
$p_0\delta<2$ and $2p_0(1-\delta)/(2-p_0\delta)< n$. Then by the H\"{o}lder inequality and volume comparison under $\Ric\ge -k$,


\begin{eqnarray*}
\left(\int_{B_{g_0}(x,r)}w^{p_0}\;d\mathrm{vol}_{g_0}\right)^{\frac{1}{p_0}}&\le&
\left(\int_{B_{g_0}(x,r)}w^{2}\;d\mathrm{vol}_{g_0}\right)^{\frac{\delta}{2}}\left(\int_{B_{g_0}(x,r)}w^{\frac{2p_0(1-\delta)}{2-p_0\delta}}\;d\mathrm{vol}_{g_0}\right)^{\frac{2-p_0\delta}{2p_0}}\\
&\le&\left(\int_{B_{g_0}(x,r)}w^{2}\;d\mathrm{vol}_{g_0}\right)^{\frac{\delta}{2}}\left(\int_{B_{g_0}(x,r)}w^{n}\;d\mathrm{vol}_{g_0}\right)^{\frac{(1-\delta)}{n}}\\
&&\times\mathrm{Vol}_{g_0}(B_{g_0}(x,r))^{\frac{[n-2p_0(1-\delta)/(2-p_0\delta)](2-p_0\delta)}{2np_0}}\\
&\le& C_1\left(\int_{B_{g_0}(x,r)}w^{2}\;d\mathrm{vol}_{g_0}\right)^{\frac{\delta}{2}}\left(\int_{B_{g_0}(x,r)}w^{n}\;d\mathrm{vol}_{g_0}\right)^{\frac{(1-\delta)}{n}},
\end{eqnarray*}

where $C_1=C_1(n,k)\ge 1$. Thanks to \eqref{FIE},

\be\label{p=1 case}\begin{split}
\left(\int_{B_{g_0}(x,\theta r)}w^{n}\;d\mathrm{vol}_{g_0}+1\right)^{\frac{1}{n}}&\le\quad  \frac{C_{\infty}}{r^{\alpha_1}(1-\theta)^{\alpha_2}}\left(\int_{B_{g_0}(x,r)}w^{2}\;d\mathrm{vol}_{g_0}+1\right)^{\frac{\delta}{2}}\\
&\quad\quad\times 
\left(\int_{B_{g_0}(x,r)}w^{n}\;d\mathrm{vol}_{g_0}+1\right)^{\frac{1-\delta}{n}},
\end{split}
\ee

where 

\bee
 \alpha_1:=\sum_{i=0}^{\infty}\frac{1}{(1-\beta)p_0\mu^{i}}, \alpha_2:=\sum_{i=0}^{\infty}\frac{i+1}{(1-\beta)p_0\mu^{i}} \text{  and  } C_{\infty}:=2^{2+\alpha_2}C^{\sum_{i=0}^{\infty}\frac{1}{p_0\mu^{i}}}C_1.
\eee

We then choose $\theta$ and $r$ as in \cite[Chapter 19]{Li2012}. Let 
\[
\tau_{-1}:=\frac{1}{2},\quad \tau_{i-1}:=\frac{1}{2}+\sum_{j=1}^{i}\left(\frac{1}{2}\right)^{j+1}, \text{  for  } i\ge 1.
\]
$\gamma_i:=\tau_{i-1}/\tau_i\in (0,1)$ for $i\ge 0$ and $\tau_{i-1}$ is a strictly increasing sequence. Moreover, $\tau_{i-1}\in[1/2,1)$ and 
\[
1-\gamma_i=\frac{\tau_i-\tau_{i-1}}{\tau_i}\ge 2^{-(i+2)}.
\]
We choose $r=\tau_i$ and $\theta=\gamma_i$ in \eqref{p=1 case}
\bee
\begin{split}
\left(\int_{B_{g_0}(x,\tau_{i-1})}w^{n}\;d\mathrm{vol}_{g_0}+1\right)^{\frac{1}{n}}&\le\quad  \frac{C_{\infty}A_0^{\delta}}{\tau_i^{\alpha_1}(1-\gamma_i)^{\alpha_2}} 
\left(\int_{B_{g_0}(x,\tau_i)}w^{n}\;d\mathrm{vol}_{g_0}+1\right)^{\frac{1-\delta}{n}}\\
&\le 2^{\alpha_1}2^{\alpha_2(i+2)}C_{\infty}A_0^{\delta}\left(\int_{B_{g_0}(x,\tau_i)}w^{n}\;d\mathrm{vol}_{g_0}+1\right)^{\frac{1-\delta}{n}},
\end{split}
\eee
where $A_0:=\left(\int_{B_{g_0}(x,1)}w^{2}\;d\mathrm{vol}_{g_0}+1\right)^{\frac{1}{2}}$. Iterating the above inequality in $i$, we have
\bee
\begin{split}
\left(\int_{B_{g_0}(x,\tau_{-1})}w^{n}\;d\mathrm{vol}_{g_0}\right)^{\frac{1}{n}}&\le 2^{\alpha_1/\delta}C_{\infty}^{1/\delta}2^{\alpha_2\sum_{j=0}^{\infty}(j+2)(1-\delta)^{j}}A_0, 
\end{split}
\eee
here we have used the following inequalities:
\bee
\begin{split}
1&\le \liminf_{i\to\infty}\left(\int_{B_{g_0}(x,\tau_i)}w^{n}\;d\mathrm{vol}_{g_0}+1\right)^{\frac{(1-\delta)^{i}}{n}}\\
&\le \lim_{i\to\infty}\left(\int_{B_{g_0}(x,1)}w^{n}\;d\mathrm{vol}_{g_0}+1\right)^{\frac{(1-\delta)^{i}}{n}}=1.
\end{split}
\eee
We proved the $L^{n}$ norm bound for $\na^2\rho$ over balls of radius $1/2$. The estimate on balls of radius $1$ then follows from Lemma \ref{ap-rho-L2}, the Ricci curvature lower bound and a covering argument. This completes the proof of the proposition for $n\ge 5$.
\end{proof}

We end this section by using $\rho$ to establish the density of Sobolev space $W^{2,p}_0(M)$ in $W^{2,p}(M)$. This is motivated by the works \cite{ImperaRimoldiVeronelli2022,HondaMariRimoldiVeronelli}. As pointed out in \cite[Theorem 1.9]{HondaMariRimoldiVeronelli}, the density is not necessarily true in general. 

\begin{thm}\label{densityThm}
For any $n\geq 5$ and $C_S>0$, there is 
$\delta(n,C_S)>0$ such that if $(M,g)$ is a complete non-compact Riemannian manifold with
\begin{enumerate}
\item[(i)] For any $\varphi\in C^\infty_{c}(B_{g}(x,1))$,
$$\left(\int_{B_{g}(x,1)} |\varphi|^\frac{2n}{n-2} \; d\mathrm{vol}_{g}\right)^\frac{n-2}{n}\leq C_S\int_{B_{g}(x,1)}|\nabla \varphi|^2+\varphi^2 \; d\mathrm{vol}_{g};$$
\item[(ii)] $\displaystyle \left(\int_{B_{g}(x,1)}|\Rm|^{n/2} \; d\mathrm{vol}_{g}\right)^{2/n}\leq \delta$;
\item[(iii)] $\Ric\geq -k$ for some $k\in \mathbb{R}$;
\item[(iv)] 
There are $x_0\in M$ and constant $C>0$ such that for all $r>0$,
$$\mathrm{Vol}_g\left(B_g(x_0,r) \right)\leq Cr^n.$$
\end{enumerate}
then $ W^{2,p}(M)=W^{2,p}_0(M)$ for all $n>p\geq 2$. 
\end{thm}
\begin{proof}
By \cite{GGP2017}, it suffices to show that for any $f\in C^\infty\cap W^{2,p}$, we can find $f_i\in C^\infty_{loc}$ so that $f_i\to f$ in $W^{2,p}$. 

Fix $x_0\in M$ and let $R>0$ be large. Let $\rho$ be the function obtained from  Theorem~\ref{Thm:exhaustion-C2}. Let $\phi$ be a cutoff function on $[0,+\infty)$ so that $\phi$ vanishes outside 
$(-\infty,10]$, identical to $1$ on 
$(-\infty, 8]$ and satisfies $0\leq -\phi'\leq 10^3$ and $|\phi''|\leq 10^3$. Define $f_R=f\cdot \Phi_R$ where $\Phi_R(x)=\phi(\frac{\rho(x)}{R})$ for $x\in M$. Clearly, $f_R$ is compactly supported on $\{x\in M: 0\leq \rho(x)\leq 10R\}\subset B_{g}(x_0,10CR+CC_1)$ 
, where $C$ and $C_1$ are the constants from Theorem~\ref{Thm:exhaustion-C2}.
We now claim that $f_R\to f$ as $R\to +\infty$ in $W^{2,p}(M)$. We only consider the highest order term since the lower order one can be treated in a similar but simpler argument. 
\begin{equation}
    \begin{split}
        \int_M |\nabla^2(f_R-f)|^p \;d\mathrm{vol}_g 
        &=\int_M |\nabla^2 ((\Phi_R-1)f ) |^p \;d\mathrm{vol}_g \\
        &\leq c(p)
        \int_M |f \nabla^2 \Phi_R|^p +2|\nabla f|^p|\nabla \Phi_R|^p+ |\Phi_R-1|^p |\nabla^2 f|^p\; d\mathrm{vol}_g\\
        &\leq \mathbf{I}+\mathbf{II}+\mathbf{III}.
    \end{split}
\end{equation}

Clearly, $\mathbf{III}=o(1)$ as $R\to+\infty$ as $\nabla^2 f\in L^p$. On the other hand, we have
\begin{equation}
    \begin{split}
        |\nabla \Phi_R|^2&=\frac{|\phi'|^2}{R^2} |\nabla \rho |^2\leq \frac{C_n}{R^2}
    \end{split}
\end{equation}
and $\mathrm{supp}\left(\nabla\Phi_R \right)\subset \{ 8R\leq  \rho\leq 10R\}\subset A_g(x_0,8R,10CR+CC_1)$. 
Hence, $\mathbf{II}=o(1)$ as $R\to +\infty$ using again $\nabla f\in L^p$. 

It remains to consider $\mathbf{I}$. Similar to the gradient estimate,
\begin{equation}
|\nabla^2 \Phi_R|=\left|\frac{\phi'}{R}\nabla^2\rho+ \frac{\phi''}{R^2}\nabla\rho \otimes \nabla\rho\right|.
\end{equation}

It suffices to handle the Hessian term as the gradient term can be estimated as before. If we let $\{x_i\}_{i=1}^N$ be a maximal $1$-net of $A_g(x_0,8R,10CR +CC_1)$ in the sense that $ \{x_i\}_{i=1}^N\subset A_g(x_0,8R,10CR +CC_1)\subset \bigcup_{i=1}^N
B_g(x_i,1)$ and $\{B_g(x_i,\frac14)\}_{i=1}^N$ are all disjoint balls in $A_g(x_0,8R,10CR+CC_1)$. We claim that $N\leq CR^n$ for some constant $C>0$ independent of $R$. If $k=0$, this follows from standard volume doubling argument. By \cite[Theorem 3.6]{Wang2018} and the equivalence of $\bar \nu$ of a fixed scale with $L^2$ Sobolev inequality \cite[Theorem 4.2.1]{ZhangBook}, the assumptions (i) and (iii) imply that for some constant $v_0>0$, $\mathrm{Vol}_g\left(B_g(x,\frac14)\right)\geq v_0$ for all $x\in M$. Hence, 
$$Nv_0\leq  \sum_{i=1}^N\mathrm{Vol}_g\left(B_g(x_i,\frac 14)\right)\leq \mathrm{Vol}_g\left(B_g(x_0,20CR)\right)\leq C'R^n.$$
This shows the claim.

By using the above claim and $L^n_{loc}$ bound of $|\nabla^2\rho|$ from Proposition~\ref{Thm:exhaustion-C2}, 
\begin{equation}
    \begin{split}
        &\quad \frac1{R^p}\int_M |f|^p|\phi'|^p|\nabla^2\rho|^p\; d\mathrm{vol}_g \\
        &\leq \frac{C}{R^p}\left(\int_{A_g(x_0,8R,10CR +CC_1)} |\nabla^2\rho|^{n}\;d\mathrm{vol}_g\right)^{p/n} \left( \int_{\{8R\leq\rho\leq 10R\}}  |f|^\frac{pn}{n-p}\;d\mathrm{vol}_g\right)^\frac{n-p}{n}\\
        &\leq \frac{C}{R^p} \left(\sum_{i=1}^{CR^n}\int_{B_g(x_i,1)}|\nabla^2\rho|^n \right)^{p/n} \left( \int_{\{8R\leq\rho\leq 10R\}}  |f|^\frac{pn}{n-p}\;d\mathrm{vol}_g\right)^\frac{n-p}{n}\\
        &\leq C \left( \int_{\{8R\leq\rho\leq 10R\}}  |f|^\frac{pn}{n-p}\;d\mathrm{vol}_g\right)^\frac{n-p}{n}.
    \end{split}
\end{equation}

It remains to show that $f\in  L^\frac{np}{n-p}(M)$. 
We want to show that $\int_{M}|f|^\frac{np}{n-p}\;d\mathrm{vol}_g<\infty$.
To do this, we first choose a maximal 
$1$-net $\{x_i\}_{i=1}^{\infty}$ of $M$ in the sense of above. 

It is well-known that $L^2$-Sobolev inequality implies $L^p$-Sobolev inequality for all $n>p>2$ by applying $L^2$-Sobolev to $f^{q}$ for suitable $q$. Instead of using (i), we consider the standard Sobolev embedding from \cite[Theorem 14.3]{Li2012} together with the volume lower bound obtained above. In particular, for all $u\in C^\infty_{c}(B_g(x,2))$,
$$||u||_{L^\frac{np}{n-p}(B_g(x,2))}\leq 
C_{n,p,k, C_S}||\nabla u||_{L^p(B_g(x,2))}.$$

By expressing 
$|f|^p=\sum_{i=1}^{\infty} |f|^p\phi_i$ where $\{\phi_i\}_{i=1}^{N_1}$ is a fintie sequence of partition functions to be specified later. Then we have 
\begin{equation}
    \begin{split}
   ||f||_{L^\frac{np}{n-p}(M)}^p   &= 
    \left\|\sum_{i=1}^{\infty}  (|f|^p\phi_i)\right\|_{L^\frac{n}{n-p}(M)}\\
       &\leq \sum_{i=1}^{\infty} ||f \phi_i^{1/p} ||^p_{L^\frac{np}{n-p}(B_g(x_i,2))}\\
       &\leq C\sum_{i=1}^{N_1} ||\nabla (f\phi_i^{1/p}) ||^p_{L^p(B_g(x_i,2))}\\
       &\leq C ||\nabla f||^p_{L^p(M)}+C \sum_{i=1}^{\infty} \int_{B_g(x_i,2)} |f|^p |\nabla \phi_i^{1/p}|^p \;d\mathrm{vol}_g  \\
       &\leq C||\nabla f||^p_{L^p(M)}+ C \int_{M} |f|^p \left(\sum_{i=1}^{\infty}\frac{|\nabla\phi_i|^p}{\phi_i^{p-1}} \right) \;d\mathrm{vol}_g.
    \end{split}
\end{equation}
Here we have used the fact that $\sum_{i=1}^{\infty}\phi_i=1$ to convert the sum back to the whole gradient term. Now we specify the choice of $\phi_i$ so that $\sum_{i=1}^{\infty}\frac{|\nabla\phi_i|^p}{\phi_i^{p-1}}$ is bounded uniformly. 
For each $i$, we choose a smooth function $\varphi_i$ such that $\varphi_i\equiv 1$ on $B_g(x_i,1)$, vanishes outside $B_g(x_i,2)$ and satisfies  $|\nabla\varphi_i|\leq C_p\varphi_i^{1-\frac1p}$. By volume comparison and Ricci lower bound, any $x\in M$ can at most sit inside $C(n,k)$ of $B_g(x_i,2)$. 

To achieve this, we define $$\phi_i(x)=\frac{\varphi_i(x)}{\sum_{j=1}^{\infty} \varphi_j(x)}$$


so that $\sum_{i=1}^{\infty}
\phi_i\equiv 1$. 
By the choice of covering, for any $x\in M$, there is $i_0$ such that $x\in B_g(x_{i_0},1)$ and hence $\sum_{j=1}^{\infty} \varphi_j(x)\geq 1$. With this choice of $\phi_i$, direct computation show that $\sum_{i=1}^{\infty}\frac{|\nabla\phi_i|^p}{\phi_i^{p-1}}\leq C(n,k)$. 
We conclude that for $f\in W^{1,p}(M)$,
\be
||f||_{L^\frac{np}{n-p}}\leq C'\left( ||f||_{L^p}+||\nabla f||_{L^p}\right).
\ee
This completes the proof.

\end{proof}

Finally, we remark that if in addition we assume a slightly stronger bound on $\Ric$, it is clear from the proof that we might improve estimates of $\nabla^2\rho$ accordingly. Furthermore, if we assume uniform $L^p$ bound on $\Ric$ for sufficiently large $p$, then a consequence of Ricci flow smoothing will yield the existence of $C^2$ distance like function $\rho$, see \cite{Yang1992a}.

\section{pseudolocality of Ricci flow}

We will need a pseudo-locality of Ricci flow under \textit{bounded curvature} which is crucial in estimating the life-span of the Ricci flows. The following pseudo-locality theorem is a modified version from the work in \cite{ChanChenLee2021} which was proved by the first, third author and Chen.
\begin{thm}\label{pseudo}
For all $n\geq 3$ and $A,\lambda,\tau,v_0>0$, there are $C_0$, $\sigma$ and $\hat T>0$ depending only on $n,\lambda,\tau,A,v_0$ such that the following holds. Suppose $(N^n,g(t))$ is a complete Ricci flow of bounded curvature on $[0,T]$ and the initial metric $g(0)$ satisfies the followings:
\begin{enumerate}
\item[(a)] $\mathcal{R}(g(0))\geq -\lambda;$
\item[(b)] $\mathrm{Vol}_{g(0)}\left( B_{g(0)}(p,2r)\right)\leq L \cdot \mathrm{Vol}_{g(0)}\left( B_{g(0)}(p,r)\right)\leq Lv_0r^n$ for all $r\leq 1$; 
\item[(c)]  $\nu(B_{g(0)}(p,1),g(0),\tau)\geq -A$;
\item[(d)] $\displaystyle\left(\int_{B_{g(0)}(p, 1)}|\Rm(g_0)|^{n/2}d\mathrm{vol}_{g(0)}\right)^{2/n}\leq \e $ for some $\e<\sigma $,
\end{enumerate}
for all $p\in M$. Then we have for any $x$ $\in N$ and  $t$ $\in (0,  T\wedge \hat T]$,
\be\label{C0 bdd for Riem}
|\Rm|(x,t)\leq\frac{C_0\e}{t} \quad\text{and}\quad \mathrm{inj}_{g(t)}(x) \geq C_0^{-1}\sqrt{t}.
\ee
Moreover, we have $\displaystyle\left(\int_{B_t(x,1)}|\Rm(g(t))|^{n /2}d\mathrm{vol}_{g(t)}\right)^{2/n}\leq C_0 \e$ for all $x\in N$. In particular, the Ricci flow must exist up to $\hat T$.
\end{thm}
\begin{proof}
This is proved by an identical argument of \cite[Theorem 1.2]{ChanChenLee2021}. As pointed out in  \cite[Remark 1.3]{ChanChenLee2021}, the Ricci lower bound was only used for the purpose of scalar curvature lower bound, uniform volume doubling up to a fixed scale (which follows from volume comparison in the presence of Ricci lower bound) and volume ratio upper bound so that we can invoke \cite[Lemma2.2]{LeeTam2022} to compare geodesic balls at different time slice uniformly (after scaling), see Lemma~\ref{DistanceDistortion-RF} for the precise statement. The scale on the entropy is unimportant as we can scale up the metric. The curvature concentration on the ball of same radius is based on covering argument from uniform volume doubling. 
\end{proof}

In \cite{ChanChenLee2021}, the assumptions on the scalar curvature and the volume ratio are simple consequences of Ricci lower bound which is natural to be considered. This is however difficult to be preserved when constructing approximating Ricci flow solution. The flexibility on the newly stated pseudolocality 
is crucial in constructing Ricci flows with unbounded curvature.

\section{Existence of Ricci flow}
In this section, we will construct the Ricci flow by using the pseudolocality and property of the exhaustion function $\rho$. Since $g_0$ is not necessarily of bounded curvature, we use a trick of Topping \cite{Topping2010} (see also \cite{Hochard2016}) to construct local Ricci flows using  the function $\rho$ constructed in Section~\ref{Sec:exhaustion}.  We follow the treatment in \cite{LeeTam2020}.

Let $\kappa\in (0,1)$, $f:[0,1)\to[0,\infty)$ be the function:
\be\label{e-exh-1}
 f(s)=\left\{
  \begin{array}{ll}
    0, & \hbox{$s\in[0,1-\kappa]$;} \\
    -\displaystyle{\log \lf[1-\lf(\frac{ s-1+\kappa}{\kappa}\ri)^2\ri]}, & \hbox{$s\in (1-\kappa,1)$.}
  \end{array}
\right.
\ee
Let   $\varphi\ge0$ be a smooth function on $\R$ such that $\varphi(s)=0$ if $s\le 1-\kappa+\kappa^2 $, $\varphi(s)=1$ for $s\ge 1-\kappa+2 \kappa^2 $
\be\label{e-exh-2}
 \varphi(s)=\left\{
  \begin{array}{ll}
    0, & \hbox{$s\in[0,1-\kappa+\kappa^2]$;} \\
    1, & \hbox{$s\in (1-\kappa+2\kappa^2,1)$.}
  \end{array}
\right.
\ee
such that $\displaystyle{\frac2{ \kappa^2}}\ge\varphi'\ge0$. Define
 $$\mathfrak{F}(s):=\int_0^s\varphi(\tau)f'(\tau)d\tau.$$
 \bigskip

The following highlights the important properties of $\mathfrak{F}$.
\begin{lma}[Lemma 4.1 in \cite{LeeTam2020}] \label{l-exhaustion-1}Suppose   $0<\kappa<\frac18$. Then the function $\mathfrak{F}\ge0$ defined above is smooth and satisfies the following:
\begin{enumerate}
  \item [(i)] $\mathfrak{F}(s)=0$ for $0\le s\le 1-\kappa+\kappa^2$.
  \item [(ii)] $\mathfrak{F}'\ge0$ and for any $k\ge 1$, $\exp( -k\mathfrak{F})\mathfrak{F}^{(k)}$ is uniformly  bounded.
  \item [(iii)]  For any $ 1-2\kappa <s<1$, there is $\tau>0$ with $0<s -\tau<s +\tau<1$ such that
 \bee
 1\le \exp(\mathfrak{F}(s+\tau)-\mathfrak{F}(s-\tau))\le (1+c_2\kappa);\ \ \tau\exp(\mathfrak{F}(s-\tau))\ge c_3\kappa^2
 \eee
  for some absolute constants  $c_2>0, c_3>0$.
\end{enumerate}

\end{lma}

Now we are ready to prove the main result.
\begin{proof}[Proof of Theorem~\ref{main-Existence}]
We follow and modify the argument in \cite{He2016}. Fix $p\in M$ and $\rho$ be the smooth proper function obtained from Theorem~\ref{Thm:exhaustion-C2}. For any  $R>0$ sufficiently large, let  $U_{R}$ be the component of $ \{x|\ \rho(x)<R\}$ which contains $p$.  In this way, $U_{R}$ will exhaust $M$ as $R\to\infty$.  On each $U_R$, define 
$$F_R(x)=\mathfrak{F}\left(\frac{\rho(x)}{R} \right),\quad\text{and}\quad g_{R,0}=e^{2F_R}g_0$$
which is a complete metric on $U_R$ with bounded curvature by \cite[Lemma 4.3]{LeeTam2020}, see also \cite{Hochard2016}. It will be sufficient to fix a small $\kappa$.

By the celebrated work of Shi \cite{Shi1989}, we may deform each $g_{R,0}$ using Ricci flow for a short time. Let $g_R(t), t\in [0,T_R)$ be the Shi's Ricci flow on $U_R$ with $g_R(0)=g_{R,0}$ and $T_R$ is the maximal existence time.  Our goal is to apply Theorem~\ref{pseudo} on each $(U_R,g_R(t)),t\in [0,T_R)$ to show that $T_R> \hat T$ with uniform scaling invariant estimates on $(0,\hat T]$.  In other word, it suffices to show that $g_{R,0}$ satisfies the assumptions in Theorem~\ref{pseudo} after an uniform re-scaling. If this is the case, then $T_R>\hat T$ which is uniform in $R\to +\infty$. Since $g_{R,0}=g_0$ on any compact subset $\Omega\Subset M$ as $R\to +\infty$. By \cite[Corollary 3.2]{Chen2009} (see also \cite{Simon2008}) and the modified Shi's higher order estimate \cite[Theorem 14.16]{ChowBookII}, we infer that for any $k\in \mathbb{N}$ and $\Omega\Subset M$, we can find $C(n,k,\Omega,g_0)>0$ so that for all $R\to +\infty$,
\begin{align}
\sup_{\Omega\times [0,\hat T]}|\nabla^k \mathrm{Rm}(g_R(t))|\leq C(n,k,\Omega,g_0).
\end{align}
By working on coordinate charts and  Ascoli-Arzel\`a Theorem, we may pass to a subsequence to obtain a smooth solution $g(t)=\lim_{R\rightarrow +\infty}g_R(t)$ of the Ricci flow on $M\times [0,\hat T]$ with $g(0)=g_0$ so that the estimates in conclusion holds. Moreover, it is a complete solution by \cite[Corollary 3.3]{SimonTopping2016}. The scalar curvature lower bound is a simple consequence of maximum principle, either by applying it on the approximating Ricci flow $g_R(t)$ or using Chen's maximum principle \cite{Chen2009} on $g(t)$.

The remaining part of the proof will be devoted to show that $g_{R,0}$ satisfies the assumptions in Theorem~\ref{pseudo} uniformly for $R$ sufficiently large. In what follows, we will use $\Lambda$ to denote any constants which might depend on the (ineffective) Ricci lower bound. We will show that the error in this form can be ignored eventually by letting $R\to +\infty$.

\begin{claim}[Uniform scalar curvature lower bound]\label{claim1}
There is $R_0>0$ such that for all $R>R_0$, we have on $U_R$, 
$$\mathcal{R}(g_{R,0})\geq -2\lambda.$$
\end{claim}
\begin{proof}[Proof of Claim]
By the conformal change formula of scalar curvature, we have for $R$ sufficiently large,
 \begin{equation}
 \begin{split}
 &\quad  \mathcal{R}(g_{R,0})-e^{-2\mathfrak{F}}\mathcal{R}(g_{0})\\
  &=e^{-2\mathfrak{F}}\left(-\frac{4(n-1)}{n-2}\left(\frac{(n-2)^2}{4R^2}|\mathfrak{F}'|^2|\nabla\rho|^2+\frac{n-2}{2R^2}\mathfrak{F}''|\nabla\rho|^2+\frac{n-2}{2R}\mathfrak{F}'\Delta\rho\right)\right)\\
  &\geq -\frac{C_n}{R} \Lambda \\
  &\geq -\lambda
  \end{split}
\end{equation}
where $\nabla$ denotes the connection with respect to $g_0$. Here we have used \ref{Thm:exhaustion-C2} and Lemma \ref{l-exhaustion-1} to control $\mathfrak{F}$, $|\Delta \rho|$ and $|\nabla \rho|$ point-wisely. Since $F_R\geq 0$, 
$$\mathcal{R}(g_{R,0})\geq e^{-2\mathfrak{F}}\mathcal{R}(g_{0})-\lambda \geq -2\lambda.$$
This completes the proof.
\end{proof}

\begin{claim}[Uniform volume doubling]\label{claim2}
There exists $R_0>0$ such that for all $R>R_0$, for all $x\in U_R$ and $r\leq \frac12$, we have
$$\mathrm{Vol}_{g_{R,0}}\left( B_{g_{R,0}}(x,2r)\right)\leq 2L^2 \cdot \mathrm{Vol}_{g_{R,0}}\left( B_{g_{R,0}}(x,r)\right)\leq 2L^2(1+c_2\kappa)^n v_0 r^n.$$
\end{claim}
 \begin{proof}[Proof of Claim]
 The argument is similar to that in \cite{He2016, LeeTam2017, Huang2019}. We include a sketch here. Let $x\in U_R$ and $r\leq \frac12$. If $x\in U_R$ is such that $\rho(x)\leq (1-2\kappa)R$, then if $z\in B_{g_0}(x,1)$,
 \begin{equation}
     \begin{split}
         \rho(z)&\leq \rho(x)+ \Lambda d_{g_0}(x,z)\\
         &\leq (1-2\kappa) R + \Lambda \\
         &\leq (1-\kappa+\kappa^2)R
     \end{split}
 \end{equation}
 provided that $R$ is sufficiently large. Hence Lemma~\ref{l-exhaustion-1} implies that $g_{R,0}=g_0$ on $B_{g_0}(x,1)$. In particular, the volume doubling and volume ratio upper bound on ball centred at $x$ trivially holds. 
 
 If instead $R>\rho(x)\geq (1-2\kappa) R$, argue as in \cite{LeeTam2017}, (iii) in Lemma~\ref{l-exhaustion-1} imply that as long as $R$ is sufficiently large, we will have 
 $B_{g_{R,0}}(x,1)\subset \{z\in U_R:s-\tau <R^{-1}\rho(z)<s+\tau \}$ where $s=\frac{\rho(x)}{R}$ and $\tau(s)$ is the uniform constant from Lemma~\ref{l-exhaustion-1} depending only on $s$. Therefore we have 
 \begin{equation}\label{metric-com}
      e^{2\mathfrak{F}(s-\tau)}g_0\leq g_{R,0}\leq e^{2\mathfrak{F}(s+\tau)}g_0
 \end{equation}
 on $B_{R,0}(x,1)$. And hence, for $r\leq \frac12$ (so that $2r\leq 1$),
\begin{equation}   \begin{split} 
\mathrm{Vol}_{g_{R,0}}B_{g_{R,0}}(x,r)\geq&\mathrm{Vol}_{g_{R,0}}B_{g_{0}}(x, e^{-\mathfrak{F}(s+\tau)}r)\\
\geq&e^{n\mathfrak{F}(s-\tau)}L^{-2}\mathrm{Vol}_{g_{0}}B_{g_{0}}(x, e^{-\mathfrak{F}(s+\tau)}4r)\\
\geq&e^{n[\mathfrak{F}(s-\tau)-\mathfrak{F}(s+\tau)]}L^{-2}\mathrm{Vol}_{g_{R,0}}B_{g_{0}}(x, e^{-\mathfrak{F}(s+\tau)}4r)\\
    \geq& \frac1{(1+c_2\kappa)^n}L^{-2}\mathrm{Vol}_{g_{R,0}}B_{g_{R,0}}(x, e^{\mathfrak{F}(s-\tau)-\mathfrak{F}(s+\tau)}4r)\\
    \geq& \frac12 L^{-2}\mathrm{Vol}_{g_{R,0}}B_{g_{R,0}}(x, 2r)\\
\end{split}
\end{equation}
 provided that we choose $\kappa$ sufficiently small (independent of $R\to +\infty$). The volume ratio upper bound is similar. This completes the proof.
 \end{proof}

 \begin{claim}[Uniform entropy lower bound]\label{claim3}
 There exists $\tilde A,R_0>0$ depending only on $n,A,\tau$  such that for all $R>R_0$, for all $x\in U_R$, 
 $$ \nu\left(B_{g_{R,0}}(x,1/2),g_{R,0},\tau \right)\geq -
 \tilde A-2\tau \lambda.$$
 \end{claim}
\begin{proof}[Proof of Claim]
The proof is relatively simpler as we are only looking for a rough estimate. By Claim~\ref{claim1} and \cite[Lemma 3.1]{Wang2018}, it suffices to estimate the lower bound of
$$ \bar \nu\left(B_{g_{R,0}}(x,1/2),g_{R,0},\tau\right)$$
 decreasing $\tau_0$. As in the proof of Claim~\ref{claim2}, if $x\in U_R$ such that $\rho(x)\leq (1-2\kappa)R$, then the conclusion holds trivially as $g_{R,0}=g_0$ on $B_{g_{R,0}}(x,1)$.  Hence, it suffices to consider $R>\rho(x)>(1-2\kappa)R$ where \eqref{metric-com} holds on $B_{g_{R,0}}(x,1)$. 

For notation convenience, we denote $g_{R,0}$ by $g$, $e^{2\mathfrak{F}(s+\tau)}g_0$ by $h$ and $\Omega=B_g(x,1/2)$ so that on $\Omega$, we have $(1+\e)^{-2} h\leq g\leq h$ where $\e$ can be made as small as we wish by decreasing $\kappa$ using Lemma~\ref{l-exhaustion-1}. Here $s=\frac{\rho(x)}{R}$. 

Since $\mathfrak{F}\geq 0$, by the monotonicity of $\bar\nu$ (see \cite{Wang2018})
\begin{equation}
    \begin{split}
        \bar\nu(B_h(x,3/4),h,\tau_0)
        &=\bar\nu\left(B_{g_0}(x,\frac34 e^{-\mathfrak{F}(s+\tau)}),g_0,e^{-2\mathfrak{F}(s+\tau)}\tau \right)\\
        &\geq \bar\nu\left(B_{g_0}(x,\frac34) ,g_0,\tau \right)\\
        &\geq -A.  
    \end{split}
\end{equation}

Since $\Omega\subset B_h(x,\frac34)$, using the lower bound of $\bar \nu$, Lemma~\ref{ent to S ineq} and metric equivalence of $g$ and $h$, there exists $C_0(n,A, \tau)>0$ such that for all $f\in C^\infty_{c}(\Omega)$, 
\begin{equation}
\left(\int_\Omega |f|^\frac{2n}{n-2}\; d\mathrm{vol}_g \right)^\frac{n-2}{n}\leq C_0 \left( \int_\Omega |\nabla f|^2+f^2\; d\mathrm{vol}_g\right).
\end{equation}

Now we may apply the argument in \cite[Theorem 4.2.1]{ZhangBook} to show that the $L^2$ Sobolev constant dominates that of log-Sobolev constant for any fixed scale. Although it is stated for Sobolev inequality on closed manifold without boundary, the proof can be carried over directly.  In particular, there is $\tilde A>0$ depending only on $n,A$ such that $\bar\nu(\Omega,g,\tau)\geq -\tilde A$. This completes the proof.
\end{proof}

 \begin{claim}[Uniform curvature concentration]
 There exists $R_0>0$ such that for all $R>R_0$, for all $x\in U_R$, 
 $$\left(\int_{B_{g_{R,0}}(x,1)}|\Rm(g_{R,0})|^{n/2}\;d\mathrm{vol}_{g_{R,0}} \right)^{2/n}\leq 4\delta. $$
 \end{claim}
 \begin{proof}[Proof of claim.]
 By considering the sectional curvature under conformal change: 
 \begin{equation*}
    K^{g_{R,0}}_{ij}=e^{-2F_R}(K^{g_{0}}_{ij}-\sum\limits_{k\neq i,j}|\nabla_k F_R|^2+\nabla_i\nabla_i F_R+\nabla_j\nabla_j F_R)
\end{equation*}
where $\{e_i\}$ denotes an orthonormal frame with repect to $g_0$, we see that 
\begin{equation}
\begin{split}
    |\Rm(g_{R,0})|_{g_{R,0}}&\leq e^{-2F_R}|\Rm(g_{0})|_{g_{0}}+C_ne^{-2F_R}|\nabla F_R|_{g_{0}}^2+C_ne^{-2F_R}|\nabla^2 F_R|_{g_{0}}\\
    &\leq e^{-2F_R}|\Rm(g_{0})|_{g_{0}}+\frac{C_n}{R^2}e^{-2F_R}\left(|\mathfrak{F'}|^2+|\mathfrak{F''}|\right)|\nabla \rho|^2+\frac{C_n}{R}e^{-2F_R}|\mathfrak{F}'||\nabla^2 \rho|_{g_{0}}\\
    &\leq e^{-2F_R}|\Rm(g_{0})|_{g_{0}}+\frac{\Lambda}{R^2}+\frac{C_n}{R} e^{-F_R} |\nabla^2\rho|
\end{split}
\end{equation}
where we have used Lemma~\ref{l-exhaustion-1}.

Hence, 
\begin{equation}
    \begin{split}
        \int_{B_{g_{R,0}}(x,1)}|\Rm(g_{R,0})|^{n/2}\;d\mathrm{vol}_{g_{R,0}}  
        &\leq C_n \int_{B_{g_{R,0}}(x,1)} |\Rm(g_0)|_{g_0}^{n/2}\;d\mathrm{vol}_{g_{0}} \\
        &\quad +\frac{\Lambda}{R^n} \mathrm{Vol}_{g_{R,0}}\left(B_{g_{R,0}}(x,1) \right)\\
        &\quad +\frac{C_n}{R^{n/2}}\int_{B_{g_{R,0}}(x,1)} e^{nF_R/2} |\nabla^2\rho|^{n/2}_{g_0} \;d\mathrm{vol}_{g_{0}} \\
        &=\mathbf{I}+\mathbf{II}+\mathbf{III}.
    \end{split}
\end{equation}
 
 Since $B_{g_{R,0}}(x,1)\subset B_{g_0}(x,1)$, it is clear that $\mathbf{I}\leq 
 \, C_n \delta^{n/2}$ by our assumption. For $\mathbf{II}$, as in the proof of Claim~\ref{claim2}, we may assume $x\in U_R$ such that $R>\rho(x)\geq (1-2\kappa)R$ and hence for $R\geq R_0$, we have \eqref{metric-com} on $B_{g_{R,0}}(x,1)$. In particular, for $s=R^{-1}\rho(x)$,
 \begin{equation}
     \begin{split}
         \mathbf{II}&\leq
         \frac{\Lambda}{R^n}\cdot e^{n\mathfrak{F(s+\tau)}}\mathrm{Vol}_{g_{0}}\left(B_{g_{R,0}}(x,1) \right)\\
         &\leq  \frac{\Lambda}{R^n}\cdot e^{n\mathfrak{F(s+\tau)}}\mathrm{Vol}_{g_{0}}\left(B_{g_{0}}(x,e^{-\mathfrak{F(s-\tau)}}) \right)\\
         &\leq \frac{\Lambda v_0}{R^n}\cdot e^{n(\mathfrak{F(s+\tau)}-\mathfrak{F(s-\tau)})}\\
         &\leq \frac{\Lambda v_0 (1+c_2\kappa)^n}{R^n}\leq \delta
     \end{split}
 \end{equation}

 if we further increase $R_0$. Here we have used volume comparison to control the volume of unit ball. 
 
 It remains to consider $\mathbf{III}$. Using again \eqref{metric-com}, volume comparison,  Theorem~\ref{Thm:exhaustion-C2}  and Lemma~\ref{l-exhaustion-1},
 \begin{equation}
     \begin{split}
       \mathbf{III}&\leq   \frac{C_n}{R^{n/2}} e^{n\mathfrak{F}(s+\tau)/2} \int_{B_{g_{0}}(x,e^{-\mathfrak{F}(s-\tau)})} |\nabla^2\rho|^{n/2}_{g_0} \;d\mathrm{vol}_{g_{0}}\\
       &\leq \frac{C_n}{R^{n/2}}  e^{n\mathfrak{F}(s+\tau)/2}\left(\int_{B_{g_{0}}(x,e^{-\mathfrak{F}(s-\tau)})} |\nabla^2\rho|^{n}_{g_0} \;d\mathrm{vol}_{g_{0}}\right)^{1/2} \cdot \sqrt{\mathrm{vol}_{g_0}\left( B_{g_0}(x,e^{-\mathfrak{F}(s-\tau)}\right)}\\
       &\leq \frac{\Lambda}{R^{n/2}} e^{\frac12 n(\mathfrak{F(s+\tau)}-\mathfrak{F(s-\tau)})}\leq \delta
     \end{split}
 \end{equation}
 if $R$ is sufficiently large and $n\geq 5$. If $n= 4$, alternatively we apply Lemma~\ref{ap-rho-L2} and volume comparison so that 
 \begin{equation}
     \begin{split}
         \mathbf{III}&\leq  \frac{C}{R^2} e^{2\mathfrak{F}(s+\tau)} \int_{B_{g_{0}}(x,e^{-\mathfrak{F}(s-\tau)})} |\nabla^2\rho|^{2}_{g_0} \;d\mathrm{vol}_{g_{0}}\\
         &\leq \frac{\Lambda}{R^2} e^{2\mathfrak{F}(s+\tau)-2\mathfrak{F}(s-\tau)}\leq \delta.
     \end{split}
 \end{equation}
 
 This completes the proof by combining $\mathbf{I},\mathbf{II}$ and $\mathbf{III}$.
 \end{proof}

 By choosing an even smaller $\delta$, we may apply Theorem~\ref{pseudo} by considering a re-scaled flow $\sigma^{-1} g_{R}(\sigma t)$ for an uniform small $\sigma>0$. Result follows by re-rescaling it back as described above. 
 \end{proof}
 
\begin{rem}
It is clear from the proof that the existence of Ricci flow can be localized. We stick with the global existence so that the geometric implication is clearer. 
\end{rem}

The small curvature concentration might look restrictive but in fact this can be achieved by scaling if $\Rm(g_0)\in L^{n/2}$. 
\begin{lma}\label{breakingscalinginvariant-Ln/2}
Suppose $(M,g)$ is a complete manifold such that $\Rm(g)\in L^{n/2}$, then for any $\e>0$ there is $1>r_0>0$ such that for all $x\in M$,
$$\int_{B_g(x,r_0)}|\Rm(g)|^{n/2}\; d\mathrm{vol}_g <\e.$$
\end{lma}
\begin{proof}
Fix $p\in M$.
Since $|\Rm(g)|\in L^{n/2}$, completeness implies that there is $R>0$ such that 
$$\int_{M\setminus B_g(p,R)}|\Rm(g)|^{n/2}\; d\mathrm{vol}_g <\e.$$

Therefore, it suffices to consider $x\in B_g(p,R+1)$. We claim that the conclusion holds on $\overline{B_g(p,R+1)}$. Otherwise, we may find $\e_0>0$, $x_i\in B_g(p,R+1)$ and $r_i\to 0$ such that 
$$\int_{B_g(x_i,r_i)}|\Rm(g)|^{n/2}\; d\mathrm{vol}_g \geq \e_0$$
which is impossible by dominated convergence Theorem since $B_g(p,R+1)$ is pre-compact. This completes the proof by combining both.
\end{proof}

As a corollary, we have the following existence result of Ricci flow. 
\begin{thm}\label{Ln/2-existence}
Suppose $(M^n,g_0),n\geq 4$ is a complete non-compact manifold such that 
\begin{enumerate}
\item[(i)] $\liminf_{z\to +\infty}\Ric(g_0)>-\infty$;
\item[(ii)] $\inf_{x\in M}\bar\nu(B_{g_0}(x,1),g_0,\tau)\geq -A$ for some $\tau,A>0$;
\item[(iii)] $\Rm(g_0)\in L^{n/2}(M,g_0)$.
\end{enumerate}
For any $\e>0$, then there is a Ricci flow $g(t)$ from $g_0$ on $M\times [0,S_\e]$ with $S_\e>0$ such that $\mathrm{inj}(g(t))\geq C\sqrt{t}$ for some $C>0$ and
\begin{equation}\label{cur-esti-RF}
  \sup_M |\Rm(g(t))|<\frac{\e}{t}. 
\end{equation}
\end{thm}
\begin{proof}
By Lemma~\ref{breakingscalinginvariant-Ln/2}, for any $\sigma>0$ there is $r_0>0$ such that for all $x\in M$, 
$$\left(\int_{B_{g_0}(x,r_0)}|\Rm(g_0)|^{n/2}\;d\mathrm{vol}_{g_0}\right)^{2/n}<\sigma.$$

Let $0<r_1<\min\{r_0,1\}$ 
be a constant to be determined. Consider the re-scaled metric $\tilde g_0=r_1^{-2}g_0$. We choose $r_1$ sufficiently small so that 
\begin{enumerate}
    \item[(a)] $\Ric(\tilde g_0)\geq -1$;
    \item[(b)] $\mathrm{Vol}_{\tilde g_0}\left( B_{\tilde g_0}(x,2r)\right)\leq 3^n\cdot \mathrm{Vol}_{\tilde g_0}\left( B_{\tilde g_0}(x,r)\right)\leq v_0 r^n$ for all $r\leq 1$;
    \item[(c)] $\inf_{x\in M}\bar \nu\left(B_{\tilde g_0}(x,1),\tilde g_0,1 \right)\geq -A$;
    \item[(d)] For all $x\in M$, 
    $$\left(\int_{B_{\tilde g_0}(x,1)}|\Rm(\tilde g_0)|^{n/2}\;d\mathrm{vol}_{\tilde g_0}\right)^{2/n}<\sigma.$$
\end{enumerate}
Here (b) follows from volume comparison, (c) follows from monotonicity of local entropy 
and (d) follows from $r_1< r_0$. By Theorem~\ref{main-Existence}, if we choose $\sigma$ sufficiently small, then there is a Ricci flow $\tilde g(t)$ starting from $\tilde g_0$ on $[0,T]$ with 
\begin{equation}
    \left\{
    \begin{array}{ll}
     |\Rm(\tilde g(t))|\leq C_0\sigma t^{-1};\\
    \mathrm{inj}_{\tilde g(t)}\geq C_0^{-1}\sqrt{t}.
    \end{array}
    \right.
\end{equation}

By re-scaling it back 
$g(t):=r_1^2 \tilde g(r_1^{-2}t)$ on $M\times [0,Tr_1^2]$, we obtain $g(t)$ with 
\begin{equation}\label{rough-estim-RF}
    \left\{
    \begin{array}{ll}
     |\Rm( g(t))|\leq C_0\sigma t^{-1};\\
    \mathrm{inj}_{ g(t)}\geq C_0^{-1}\sqrt{t}.
    \end{array}
    \right.
\end{equation}
Result follows since $\sigma$ is arbitrary.
\end{proof}
It is unclear to us whether the flow for different small $\e>0$ coincides with each other due to the lack of uniqueness under scaling invariant smoothing estimates. 

\begin{rem}
Theorem~\ref{Ln/2-existence} is not entirely a generalization of Theorem~\ref{main-Existence} since the existence time will be heavily depending on the asymptotic information of $g_0$ at infinity. In contrast, Theorem~\ref{main-Existence} is a quantitative existence result which enables us to apply to more generic settings.
\end{rem}

\section{Applications using Ricci flow existence}\label{application-from-RF}

In this section, we will apply the Ricci flow smoothing to study non-collapsed manifolds with small curvature concentration. Since we no longer consider the Ricci lower bound, we first point out that a version of distance distortion still hold in our situation. 

\begin{lma}\label{DistanceDistortion-RF}
Suppose $(M,g(t))$ is a complete Ricci flow on $M\times [0,T]$ such that $g(0)=g_0$ and for some $L>0$ and $\frac1{4(n-1)}>\a>0$,
\begin{enumerate}
    \item[(a)] $\mathcal{R}(g(t))\geq -L$;
    \item[(b)] $|\Rm(g(t))|\leq \a t^{-1}$;
    \item[(c)] $\mathrm{inj}(g(t))\geq L^{-1}\sqrt{t}$;
    \item[(d)] $\mathrm{Vol}_{g_0}\left(B_{g_0}(x,r)\right)\leq Lr^n$ for $0<r\leq 1$;
\end{enumerate}
then there exists $C_n,\tilde r(n,L),\sigma(n,L),R_1(n,L)>0$ such that for $t\in [0,T]$, we have 
\begin{enumerate}
    \item[(i)] $d_{g_0}(x,y)\leq d_{g(t)}(x,y)+C_n\sqrt{t}$;
    \item[(ii)] $d_{g_0}(x,y)\geq \sigma d_{g(t)}(x,y)$ if $\sigma^{-1}\sqrt{t}\leq d_{g_0}(x,y)\leq \tilde r$;
    \item[(iii)] $d_{g_0}(x,y)\geq \sigma d_{g(t)}(x,y)$ if $d_{g_0}(x,y)\geq R_1$ and $0\leq t\leq 1$;
    \item[(iv)]  for $x,y\in M$ such that $d_{g_0}(x,y)\leq 1$ and $0\leq t\leq  1$,
    $$\sigma  d_{g(t)}(x,y)^\frac{1}{1-2(n-1)\a}\leq d_{g_0}(x,y)\leq \sigma^{-1} d_{g(t)}(x,y)^\frac{1}{1+2(n-1)\a}.$$
\end{enumerate}
\end{lma}
\begin{proof}
The conclusion (i) follows from \cite[Lemma 8.3]{Perelman2002}, see also \cite[Corollary 3.3]{SimonTopping2016} using only scaling invariant estimate of $|\Rm(g(t))|$.

To prove (ii), we want to apply a distance estimate in \cite{LeeTam2022}, see also \cite{He2016}. Let $R_0$ and $\mu$ be the constants obtained from \cite[Lemma 2.2]{LeeTam2022}. Indeed, (iii) follows immediately without scaling here. For $x,y\in M$ so that $d_{g_0}(x,y)=r_0\leq 1$. Consider $\tilde g(t)=\lambda^2 g(\lambda^{-2}t),t\in [0,\lambda^2T]$ where $\lambda=\mu^{-1}R_0 r_0^{-1}$ so that $d_{\tilde g(0)}(x,y)=\mu^{-1}R_0$.  If we choose $\tilde r$ sufficiently small, we have $\lambda\geq 1$ and note that the scalar curvature of $\tilde g(t)$ becomes better. Hence  applying \cite[Lemma 2.2]{LeeTam2022} to $\tilde g(t)$, we have  for all $t\in [0,\lambda^2 T\wedge 1]$,
$$d_{\tilde g_0}(x,y)=\mu^{-1}R_0\geq 2\mu^{-1} d_{\tilde g(t)}(x,y).$$
Re-scaling it back yields $d_{g_0}(x,y)\geq 2\mu^{-1}d_{g(t)}(x,y)$ for $t\in [0,T\wedge \lambda^{-2}]$ which means for $t\leq \mu^2R_0^{-2}r_0^2$. This proves (ii).

The conclusion (iv) is similar to the argument in \cite[Lemma 2.4]{HuangKongRongXu2020}.  By condition (b), for $0<s\leq t\leq T$, we have  \begin{equation}\label{distance-ts}
    \left(\frac{t}{s}\right)^{-(n-1)\alpha}\leq \frac{d_{g(t)}(x, y)}{d_{g(s)}(x, y)}\leq \left(\frac{t}{s}\right)^{(n-1)\alpha}.
\end{equation}

Combining with (i), we have \begin{equation}\label{distance-right}
    d_{g_0}(x, y)\leq \left(\frac{t}{s}\right)^{(n-1)\alpha} d_{g(t)}(x, y)+C_n\sqrt{s}.
\end{equation}

Now we may assume $d_{g(t)}(x, y)\leq 1$ and choose $s$ such that $$\left(\frac{t}{s}\right)^{(n-1)\alpha} d_{g(t)}(x, y)=\sqrt{\frac{s}{t}}.$$ And hence \eqref{distance-right} implies \begin{equation}
    d_{g_0}(x, y)\leq (1+C_n\sqrt{t})d^{\frac{1}{1+2(n-1)\alpha}}_{g(t)}(x,y)\leq C_n d^{\frac{1}{1+2(n-1)\alpha}}_{g(t)}(x,y).
\end{equation}
provided that $t\leq 1$. This gives right hand side of (iv).

To see left hand side of (iv), fix $x,y\in M$ and $t\in (0,T\wedge 1]$. We may assume $t<\tilde r^2$ thanks to assumption (b). If $\sigma^{-1}\sqrt{t}\leq d_{g_0}(x,y)$, then the conclusion follows from (ii) by adjusting the choice of $\sigma$. Therefore, we might assume $\sqrt{t}>\sigma d_{g_0}(x,y)=\sqrt{s}$. By (b) and (ii), 
\begin{equation}
    \begin{split}
        d_{g(t)}(x,y)&\leq s^{-(n-1)\a} d_{g(s)}(x,y)\\
        &\leq \sigma^{-1}s^{-(n-1)\a} d_{g_0}(x,y)\\
        &=C(\sigma,n,\a) d_{g_0}^{1-2(n-1)\a}(x,y).
    \end{split}
\end{equation}
This completes the proof by further adjusting $\sigma$ if necessary.
\end{proof}

With Lemma~\ref{DistanceDistortion-RF} in mind, we can now prove Theorem~\ref{Euc-gap} using the same argument in \cite[Theorem 1.2]{ChanChenLee2021} by utilizing the new existence result,  Theorem~\ref{main-Existence}.

\begin{proof}[Proof of Theorem~\ref{Euc-gap}]The proof is identical to that of \cite[Corollary 1.1]{ChanChenLee2021}, we sketch it for reader's convenience. Since the assumptions are scaling invariant, by applying Theorem~\ref{main-Existence} on $R^{-2}g_0$ for $R\to +\infty$ and re-scale it back, we obtain a sequence of Ricci flow $g_R(t)$ on $[0,TR^2]$ with $g(0)=g_0$ for any $R>0$. Using the argument in the proof of Theorem~\ref{main-Existence}, we can extract convergent sub-sequence in locally smooth sense to obtain a long-time solution $g(t)$ to the Ricci flow with $g(0)=g_0$ 
and 
\begin{equation}
    \left\{
    \begin{array}{ll}
         |\Rm(g(t))|\leq C_0\sigma t^{-1} ; \\
         \mathrm{inj}(g(t))\geq C_0^{-1}\sqrt{t};\\
          \nu(M,g(t))\geq -A;\\
          \mathcal{R}(g(t))\geq 0.
    \end{array}
    \right.
\end{equation}
Then a contradiction type argument (see \cite[Theorem 4.2]{ChanChenLee2021}) will show that if $\sigma$ is sufficiently small, then the volume ratio and injectivity radius at time $t$ is almost Euclidean up to scale $\sqrt{t}$. Now Lemma~\ref{DistanceDistortion-RF} and scalar curvature lower bound imply that the volume ratio lower bound of $g_0$ is almost Euclidean up to scale $\sqrt{t}$. Since $t$ can be arbitrary large, this proves the almost Euclidean volume ratio. 

On the other hand, the estimates of $\mathrm{inj}(g(t))$ and Lemma~\ref{DistanceDistortion-RF} (applied to $R^{-2}g(R^2t)$) imply that $M$ can be exhausted by a sequence of Euclidean balls. The topological type follows from \cite{Brown1961}. The diffeomorphism when $n>4$ follows from the uniqueness of differentiable structure of $\mathbb{R}^n$, see \cite{Stallings1962}.
\end{proof}

As in \cite[Theorem 1.4]{ChanChenLee2021}, with Ricci flow smoothing, we may study the Gromov-Hausdorff limit of manifolds with $L^{n/2}$ curvature pinching. Thanks to the newly established existence theory, we can generalize it to situation with possibly unbounded curvature and slightly weaker curvature assumptions. We state the non-collapsing assumption in term of $\bar\nu$ in this work as it is more commonly considered outside the Ricci flow community. 

\begin{thm}\label{Theorem:GH} For any positive integer $n\geq 4$ and constant $A,L,\tau,\lambda, v_0>0$, there exists constant $\e_0$ such that the following holds. Suppose $( M^n_i,g_i, p_i)$ is a pointed sequence of complete non-compact Riemannian manifolds with the following properties: for all $x\in M$,
\begin{enumerate}
\item[(a)] $\mathcal{R}(g_i)\geq -\lambda;$
\item[(b)] $\mathrm{Vol}_{g_i}\left( B_{g_i}(x,2r)\right)\leq L \cdot \mathrm{Vol}_{g_i}\left( B_{g_i}(x,r)\right)\leq Lv_0 r^n$ for all $r\leq 1$;
\item[(c)] $\bar\nu(B_{g_i}(x,1),g_i,\tau)\geq -A$;
\item[(d)] $\displaystyle\left(\int_{B_{g_i}(x, 1)}|\Rm(g_i)|^{n/2}d\mathrm{vol}_{g_0}\right)^{2/n}\leq \e$;
\item[(e)]$\inf\Ric(g_i)>-\infty$.
\end{enumerate}
Then there exists a smooth manifold $M_\infty$ and a complete distance metric $d_\infty$ on $M_\infty$ generating the same topology as $M_\infty$ such that after passing to sub-sequence, $(M_i, d_{g_i}, p_i)$ converges in the pointed Gromov-Hausdorff sense to $(M_{\infty}, d_{g_{\infty}}, p_{\infty})$.
\end{thm}
\begin{proof}[Sketch of Proof]
The proof is identical to that of \cite[Theorem 1.4]{ChanChenLee2021}. We only give a sketch. By Theorem~\ref{main-Existence}, each $g_i$ can be deformed to $g_i(t)$ for $t\in [0,T]$ with uniform scaling invariant estimate. By Hamilton's compactness, we may assume that $(M_i,g_i(t),p_i),t\in (0,T]$ converges to $(M_\infty,g_\infty(t),p_\infty),t\in (0,T]$ in smooth Cheeger-Gromov sense. By Lemma~\ref{DistanceDistortion-RF} and a countable argument, $d_{g_i}$ converges to $d_\infty$ locally modulus diffeomorphism. Moreover, Lemma~\ref{DistanceDistortion-RF} infers that $d_\infty$ is bi-H\"older to $d_{g_\infty(t)}$. This completes the proof.
\end{proof}

\end{document}